\documentclass[a4paper,12pt,reqno]{amsart}
\usepackage[normalem]{ulem}

\usepackage{tikz}

\usepackage{graphicx}
\usepackage{subcaption}

\usetikzlibrary{shapes,positioning}
\usetikzlibrary{decorations.markings}
\usetikzlibrary{graphs}

\makeatletter
\@namedef{subjclassname@2020}{\textup{2020} Mathematics Subject Classification}
\makeatother

\usepackage{fullpage}
\usepackage{amssymb}
\usepackage{amsmath}
\usepackage{mathtools}
\usepackage{enumitem}
\usepackage{mathrsfs}
\usepackage[all]{xy}
\usepackage{mathtools}
\usepackage{hyperref}
\usepackage{url}
\usepackage{bbm}
\usepackage{longtable}
\usepackage{dsfont}
\usepackage{upgreek}

\usepackage{lmodern}
\usepackage[OT2, T1]{fontenc}

\DeclareSymbolFont{cyrletters}{OT2}{wncyr}{m}{n}
\usepackage[british]{babel}

\usepackage[margin=1in]{geometry}
\numberwithin{equation}{section} \numberwithin{figure}{section}

\DeclareMathOperator*{\Osum}{\sum{}^*}

\DeclareSymbolFont{cyrletters}{OT2}{wncyr}{m}{n}
\DeclareMathSymbol{\Sha}{\mathalpha}{cyrletters}{"58}
\DeclareMathSymbol{\Be}{\mathalpha}{cyrletters}{"42}

\renewcommand\P{\mathbb{P}}
\newcommand\Z{\mathbb{Z}}
\newcommand\N{\mathbb{N}}
\newcommand\Q{\mathbb{Q}}
\newcommand\R{\mathbb{R}}

\newcommand{\md}[1]{  \left(\textnormal{mod}\ #1\right)}
 
\renewcommand{\l}{\left}

\renewcommand{\r}{\right}
\renewcommand{\b}{\mathbf} 

\renewcommand{\c}{\mathcal} 

\renewcommand{\leq}{\leqslant}
\renewcommand{\geq}{\geqslant}
\renewcommand{\#}{\sharp}

\newcommand{\p}{\mathfrak{p}}

\newtheorem{lemma}{Lemma}

\newtheorem{theorem}[lemma]{Theorem}

\theoremstyle{definition}
\newtheorem{example}[lemma]{Example}

\newtheorem{definition}[lemma]{Definition}
\newtheorem{remark}[lemma]{Remark}

\newtheorem*{ack}{Acknowledgements}
\newtheorem*{notation}{Notation}

\numberwithin{lemma}{section}

\begin{document}

\vspace{-0,5cm}

\title
 {Averages of multiplicative functions along equidistributed sequences}

\begin{abstract}For a general family of non-negative functions matching upper and lower bounds are established for their average 
over the values of any equidistributed  sequence.\end{abstract}

\author{S. Chan} 
\address{
ISTA
 \\Am Campus 1 \\
3400 Klosterneuburg \\
Austria}
\email{stephanie.chan@ist.ac.at}

\author{P. Koymans} 
\address{Institute for Theoretical Studies\\
ETH Z\"urich \\
8092\\
Switzerland}
\email{peter.koymans@eth-its.ethz.ch}

\author{C. Pagano} 
\address{
Department of Mathematics
\\ Concordia University 
 \\ Montreal 
 \\  H3G 1M8   \\  Canada}   
\email{carlo.pagano@concordia.ca}

\author{E. Sofos} 
\address{
Department of Mathematics\\
Glasgow University  
\\ G12~8QQ \\ UK}
\email{efthymios.sofos@glasgow.ac.uk}


\maketitle

\thispagestyle{empty}

\tableofcontents

\section{Introduction}  
Averaging  multiplicative functions over integer sequences has a long history in number theory.
Nair~\cite{nair} studied the average over the values of an irreducible integer polynomial
and this was later greatly generalised by Nair--Tenenbaum~\cite{MR1618321} 
and Henriot~\cite{MR2911138}. When it comes to polynomials in two variables it   was later extended to binary forms by
La Bret\`eche--Browning~\cite{MR2276196} and to principal ideals by Browning--Sofos~\cite{ideals}.

Wolke~\cite{MR289439} had worked on averages of a multiplicative function $f\geq 0$ over the values of an increasing integer sequence, i.e. 
$$\sum_{a\in \N \cap [1,T] } f(c_a),$$ under the assumption that the sequence is `equi-distributed' along arithmetic progressions. With an eye to 
certain applications to arithmetic statistics and Diophantine equations we aim to study sums that are more general     and under weaker assumptions
on equidistribution. Omitting certain details for now, we shall work with sums of the form 
$$\sum_{a\in \mathcal A} f(c_a) \chi(c_a),$$ where $\mathcal A$ is any countable set, 
$\chi:\mathcal A \to [0,\infty)$ is any function of finite support, $c_a$ is an integer sequence, 
and 
$f$ is a non-negative arithmetic function with certain multiplicative properties.
We will give upper bounds in  Theorem~\ref{tMain} and matching lower bounds in Theorem~\ref{thm:lower}.
\subsection{The upper bound}
We introduce the necessary notation for the statement of the upper bound.
\begin{definition}
[Density functions]
\label{densfnct} 
Fix $\kappa, \lambda_1, \lambda_2,  B, K  > 0$. We define $\mathcal D(\kappa,  \lambda_1, \lambda_2, B, K)$  as the set of multiplicative functions
$h: \mathbb{N} \rightarrow \mathbb{R}_{\geq 0}$ having the properties  
\begin{itemize}
 \item for all $B <w < z$ we have 
\begin{align}
\label{elowaverg}
\prod_{\substack{p \  \mathrm{ prime} \\ w \leq p < z}} (1 - h(p))^{-1} \leq \left(\frac{\log z}{\log w}\right)^\kappa \left(1 + \frac{K}{\log w}\right)
,\end{align}
\item for every prime $p>B$ and integers $e\geq 1 $ we have 
\begin{align}
\label{eLowPower}
h(p^e) \leq \frac{B}{p},
 \end{align} 
\item for every prime $p$ and $e\geq 1 $ we have 
 \begin{align}
\label{eHighPower}
h(p^e) \leq p^{-e\lambda_1 + \lambda_2}.
\end{align}
\end{itemize}
\end{definition} 

In order to state a result that is sufficiently general but easy to use we     use the following set-up  from~\cite[\S 2.2]{MR4402657}.
Let $\mathcal A$ be an infinite set and for each $T\geq 1$ let $\chi_T:\mathcal A\to [0,\infty)$ be any function for which 
\begin{equation} \label{eq:baz1}\{a\in \mathcal A: \chi_T(a)>0 \} \ \textrm{ is finite for every } T\geq 1 .\end{equation} 
We also assume that \begin{equation} \label{eq:baz2}\lim_{T\to +\infty} \sum_{a\in \mathcal A} \chi_T(a)=+\infty .\end{equation} 
Assume that we are given a sequence of strictly positive integers 
$(c_a)_{a\in \mathcal A}$ indexed by  $\mathcal A$ and  denoted by 
$$\mathfrak C:= \{ c_a: a \in \mathcal A\}.$$
We will be interested in estimating sums of the form 
\begin{equation}\label{corectssums}\sum_{a\in \mathcal A} \chi_T(a) f ( c_a ),\end{equation} where $f$ is an arithmetic function
with the following properties: 
\begin{definition}[A class of functions]
\label{dMultRegion} Fix $ A\geq 1,\epsilon > 0, C>0$. The set $\mathcal M(A,\epsilon , C)$ of   functions
$f: \mathbb{N} \rightarrow [0,\infty)$ is defined by the property that for all coprime $m , n$ one has $$f(mn) \leq  f(m) \min\{A^{\Omega(n)}, C n^\epsilon \}.$$ 
\end{definition}
\begin{example}\label{ex:sequenc}If $c_n$ is a sequence of positive integers then 
 $$ \sum_{1\leq n \leq T  } f(c_n)$$ is of type~\eqref{corectssums} by taking
 $\mathcal A=  \N  $ and $ \chi_T(n)=\mathds 1_{[1,T]}(n)$.
\end{example}

\begin{example}
\label{ex:dilatedsets}
If $\mathcal D \subset \R^n$ is bounded and $Q(x_1,\ldots , x_n)$ an integer polynomial then 
$$ \sum_{\substack{ \b x \in \Z^n \cap T\mathcal D\\ Q(\b x )\neq 0 } } f(|Q(\b x )|)$$ is of type~\eqref{corectssums} by taking
 $\mathcal A=\{ \b x \in \Z^n:Q(\b x )\neq 0 \}$ and $ \chi_T(\b x)=\mathds 1_{T\mathcal D} (\b x)$.
\end{example}

\begin{example}
\label{ex:raitonalpoints}
If $Q_1,Q_2$ are integer polynomials in $n$ variables then 
$$ \sum_{\substack{ \b x \in (\Z\cap [-T,T])^n\\Q_1(\b x )=0,  Q_2(\b x )\neq 0 } } f(|Q_2(\b x )|)$$ is of type~\eqref{corectssums} when 
 $ \chi_T(\b x)=\mathds 1_{[0,T]}(\max |x_i|)$,
 $\mathcal A=\{\b x \in \Z^n:  Q_1(\b x )=0,  Q_2(\b x )\neq 0\}$.
\end{example}

We will need the following  notion of  `regular' distribution of the values of the integer sequence $c_a$ in arithmetic progressions. For a non-zero integer $d$ and any $T\geq 1$, let 
$$
C_d(T)=\sum_{\substack{ a\in \mathcal A \\ c_a \equiv 0 \md d }} \chi_T(a).
$$ 

\begin{definition}[Equidistributed sequences]
\label{def:levdistr} 
We say that $\mathfrak C$ is equidistributed if there exist
positive real numbers $ \theta, \xi, \kappa, \lambda_1, \lambda_2,  B, K  $ with $\max\{\theta, \xi\}<1$, 
a function $M: \mathbb{R}_{\geq 1} \rightarrow \mathbb{R}_{\geq 1}$
and  a function $h_T\in \c D ( \kappa, \lambda_1, \lambda_2,  B, K)$   
such that
\begin{align}
\label{eEquiProgression}
C_d(T) =
h_T(d)
 M(T) \Bigg\{ 1+O\Bigg(\prod_{\substack{ B<p\leq M(T) \\  p\nmid d  }} (1-
h_T(p)
 )^2\Bigg)\Bigg\} + O(M(T)^{1-\xi}) 
\end{align}
for every $T \geq 1$ and every $d\leq M(T)^\theta$, where the implied constants are independent of $d$ and $T$.
\end{definition}

It is worth emphasizing that in this definition the constants 
$ \theta, \xi, \kappa, \lambda_1, \lambda_2,  B, K  $ are all assumed to be  independent of $T$. For example, the bound $h_T(p^e)=O(1/p)$ 
in~\eqref{eLowPower} holds with an implied constant that is  independent of $e,p$ as well as  $T$.

From now on we shall abuse notation by writing $M$ for $M(T)$.
\begin{remark}\label{rem:sequ123}
The function $M(T)$ can be chosen freely in any way that makes 
$$\sum_{\substack{ a\in \mathcal A   }} \chi_T(a) = M(T) \Bigg\{ 1+O\Bigg(\prod_{\substack{ B<p\leq M(T)  }} (1-
h_T(p)
 )^2\Bigg)\Bigg\} + O(M(T)^{1-\xi}) 
$$ hold. It particular, it is necessary that it satisfies 
$$\lim_{T\to\infty} \frac{1}{M(T)} \sum_{\substack{ a\in \mathcal A   }} \chi_T(a) =1.$$  One could simply take 
$M(T):= \sum_{ a\in \mathcal A   } \chi_T(a) $, however, in certain applications it is helpful 
  to have the freedom to choose instead a smooth approximation to $\sum_{ a\in \mathcal A   } \chi_T(a) $ as a function of $T$.
\end{remark}
\begin{example}
\label{ex:sequenc123}
In the setting of Example~\ref{ex:sequenc} define $c_n=n$. Then 
$$ 
C_d(T)= \#\{1\leq n \leq T: d\mid n \}=\frac{T}{d}+O(1),
$$ 
thus, one can choose $h_T(d)=1/d$, $M(T)=T$ and $\xi=9/10$. It is important to note that the choice of $M(T)$ and $\xi $ is not unique: one may, for example, alternatively take $M(T)=T+ T^{0.4}$  and  $\xi =1/2$. 
\end{example}

We are now ready to state the main upper bound of this paper.

\begin{theorem}[The upper bound]
\label{tMain}
Let $\mathcal A$ be an infinite set and for each $T\geq 1$ define $\chi_T:\mathcal A\to [0,\infty)$ to be any function such that  both~\eqref{eq:baz1} and~\eqref{eq:baz2} hold.
Take a sequence of strictly positive integers $\mathfrak C=(c_a)_{a \in \mathcal{A}}$. Assume that $\mathfrak{C}$ is equidistributed with respect to some positive constants $\theta,\xi,\kappa,\lambda_1,\lambda_2,B,K$ and functions 
 $M(T)$  and $h_T\in \mathcal D(\kappa,  \lambda_1, \lambda_2, B, K)$   
as in Definition~\ref{def:levdistr}.
Fix any $A>1$ and assume that  $f$ is a function 
such that for every $\epsilon>0$ there exists $C>0$ for which 
$f \in \mathcal M (A,\epsilon,C)$, which is introduced in Definition~\ref{dMultRegion}.
 Assume that there exists $\alpha > 0$ and $\widetilde{B}>0$ such that for all $T\geq 1 $ one has 
\begin{equation}\label{Roll Over Beethoven}
\sup\{c_a :a\in \mathcal A, \chi_T(a)>0 \} \leq \widetilde{B} M^\alpha,
\end{equation} where $M=M(T)$ is as in Definition~\ref{def:levdistr}.
Then for all $T\geq 1 $  we have
$$
\sum_{a\in \mathcal A } \chi_T(a) f(c_a) 
  \ll  M\prod_{\substack{ B<p\leq M    } } (1-
h_T(p)
) 
\sum_{\substack{ a\leq M  }} f(a)     h_T(a),
$$
where the implied constant is allowed to depend
on $  \alpha, A, B, \widetilde{B}, \theta  , \xi, K,  \kappa$, $\lambda_i$, the function $f$ and the implied constants in~\eqref{eEquiProgression}, 
but is independent of $T$ and $M$.
\end{theorem}

\begin{remark}[Wolke's density function assumption]
\label{rem:comparison}
Note that~\cite[Assumption ${(\textrm{A}_2)}$]{MR289439} states that there exist positive constants $C_1,C_2$
such that for all $e\geq 1$ and primes $p$ one has 
$h_T(p^e)
 \leq C_1 e^{C_2} p^{-e}$. We replace this with~\eqref{eHighPower} which is a lighter assumption for large $e$.
This is of high significance in   applications 
where   $c_a$ is the sequence of values obtained by a multivariable polynomial, 
as in this case $h_T(p^e)$ is the density of zeros modulo $p^e$ and one cannot hope for a bound with $\lambda_1 \geq 1$.
\end{remark}

\begin{remark}[Wolke's level of distribution assumption]
\label{rem:comparisonlevel} Let us 
comment that~\cite[Assumption ${(\textrm{A}_4)}$]{MR289439} implies that 
$$
 C_1(T) - 
h_T(1)
 M\ll
\frac{M}{(\log M)^{D_1}}
$$ holds for every positive fixed constant $D_1$, i.e. it demands an arbitrary logarithmic saving.
Our assumption in Definition~\ref{def:levdistr} is lighter in the sense that it 
essentially only requires this for a fixed power of $\log M$.
To see this, note that when $d=1$, Definition~\ref{def:levdistr}   states that 
$$
 C_1(T) - 
h_T(1)
M
\ll
M  \prod_{  B<p\leq M   } (1-h_T(p) )^2 +  M^{1-\xi} 
 . $$ In typical applications this is of size $M/(\log M)^\kappa$, where $\kappa$ is as in~\eqref{elowaverg}.
\end{remark}

\begin{remark}
[Wolke's growth assumption]
\label{rem:comparisonfnctonmulti}Let us note that Wolke assumes that 
the function $f$ is multiplicative, which is relaxed in our work by demanding that it is submultiplicative as in 
Definition~\ref{dMultRegion}. Furthermore,~\cite[Assumption ${(\textrm{F}_1)}$]{MR289439} states that $f(p^e)$ is only allowed to grow polynomially in $e$ 
for a fixed prime $p$, whereas, Definition~\ref{dMultRegion}
relaxes this by assuming that $f(p^e)$ is allowed to grow subexponentially in $e$. 
\end{remark}

\subsection{The lower bound}
We shall see that if $f$ is not too close to $0$, then matching  lower bounds hold.
This is a  generalization of the work of Wolke~\cite[Satz 2]{MR289439}, where the main difference lies in the fact that the density functions in 
Definition~\ref{densfnct} are now allowed to grow with larger freedom.  Furthermore, Wolke's condition that $f(p^m)\geq C_0^m$ for some strictly positive real constant $C_0$ is replaced by the   more general condition~\eqref{eq:rathermoregeneral}.
\begin{theorem}[The lower bound]\label{thm:lower}Keep the notation and assumptions of Theorem~\ref{tMain}. 
Assume, in addition, that  $f:\N\to [0,\infty) $ is a multiplicative function for which 
\begin{equation} 
\label{eq:rathermoregeneral}\mathrm{ for \ each }\   L \geq 1 \ \mathrm{ one \ has } \ \inf\{f(m): \Omega(m) \leq  L  \} > 0 
\end{equation} and that the  error term in Definition~\ref{def:levdistr}  satisfies 
$$ C_d(T) = h_T(d) M(T) \left\{ 1+o_{T\to\infty}\left(\prod_{\substack{ B<p\leq M(T) \\  p\nmid d  }} (1-h_T(p)
)^2\right)\right\} + O(M(T)^{1-\xi}) $$
whevever $d\leq M(T)^\theta$. Then for all $T\geq 1 $  we have $$ \sum_{a\in \mathcal A} \chi_T(a) f(c_a) 
\gg M(T)\prod_{  p\leq M(T)    } (1-h_T(p) ) \sum_{\substack{ a\leq M(T)  }} f(a)   h_T(a) ,$$ where the implied constants are
  independent of $T$ and $M$.\end{theorem}

\begin{notation}
For a non-zero integer $m$ define $$\Omega(m):=\sum_{p\mid m } v_p(m),$$ where   $v_p$ is the standard $p$-adic valuation. Define $P^+(m)$ and $P^-(m)$ respectively to be the largest and the smallest  prime factor of a positive integer $m$ and let  $P^+(1)=1$ and $ P^-(1)=+\infty$. For a real number $x$
we reserve the notation $[x]$ for the largest integer not exceeding $x$. Throughout the paper we use 
  the standard convention that empty products are set equal to $1$. 
Throughout the paper we shall also make use of the convention that when   iterated logarithm functions $\log t,\log \log t$, etc., are used, the real variable $t$  is assumed to be sufficiently  large   to make the iterated logarithm well-defined.

The following  constants and functions are    recurring         throughout the paper:
\begin{center}
\renewcommand*{\arraystretch}{1.05}
    \begin{longtable}{| l | l | l | l |}
    \hline
   \texttt{Symbols} &  \texttt{First appearance}  \\ \hline
  $F:\mathbb N \to [0,\infty)$  &  Lemma~\ref{lemshiu}  \\ \hline 
$c_0,c_1,c_2$  &  Lemma~\ref{lemshiu}  \\ \hline 
$C,C'$  &    Lemma~\ref{lem:postvenice} \\ \hline 
$G:\mathbb N \to [0,\infty)$ &  Lemma~\ref{lem:bwv1046} \\ \hline 
 $\kappa,\lambda_1,\lambda_2, B,K$ &  Definition~\ref{densfnct} \\ \hline 
$h:\mathbb N \to [0,\infty)$ &  Definition~\ref{densfnct} \\ \hline 
$\mathcal A, T, \chi_T$ &  Equations~\eqref{eq:baz1}-\eqref{eq:baz2} \\ \hline 
$C_d(T), M, M(T)$ &  Definition~\ref{def:levdistr}  \\ \hline 
$\theta, \xi$ &  Definition~\ref{def:levdistr}  \\ \hline 
$A, \mathcal M(A,\epsilon , C)$ &  Definition~\ref{dMultRegion}  \\ \hline 
$\alpha$, $\tilde{B}$ &  Equation~\eqref{Roll Over Beethoven} \\ \hline 
$\eta_1, \eta_2, \eta_3$ &  Equation~\eqref{chapelofgouls} \\ \hline 
$Z$ &  Equation~\eqref{def:immortal rites mMA}  \\ \hline 
$b_a,c_a,d_a$ &  Equations~\eqref{eq:oboe}-\eqref{eGCD}  \\ \hline 
  \end{longtable}
  \vspace{-2em}
\end{center} 
\end{notation}

\begin{ack}Part of the work in \S\S\ref{s:blackadder}-\ref{s:thank msri for not letting me go to the workshop that started the project} was 
completed while ES and CP were at Leiden University in $2016$.
The work in the remaining sections started during the research stay of SC, PK and CP during the 
workshop \textit{Probl\`emes de densit\'e en Arithm\'etique} at CIRM Luminy in $2023$. 
We would like to thank  the organisers Samuele Anni, Peter Stevenhagen and Jan Vonk.
PK gratefully acknowledges the support of Dr.\ Max R\"ossler, the Walter Haefner Foundation and the ETH Z\"urich Foundation.
Part of the work of SC was supported by the National Science Foundation under Grant No.~\texttt{DMS-1928930}, while the 
author was in residence at the MSRI in Spring 2023.\end{ack}

\newpage 
\section{Preliminary lemmas}\label{s:blackadder}
The present section  consists of a series of preparatory lemmas that will later be used  to prove Theorem~\ref{tMain}. 
The       lemmas that do   not rely on sieve theory are structured as follows:
 
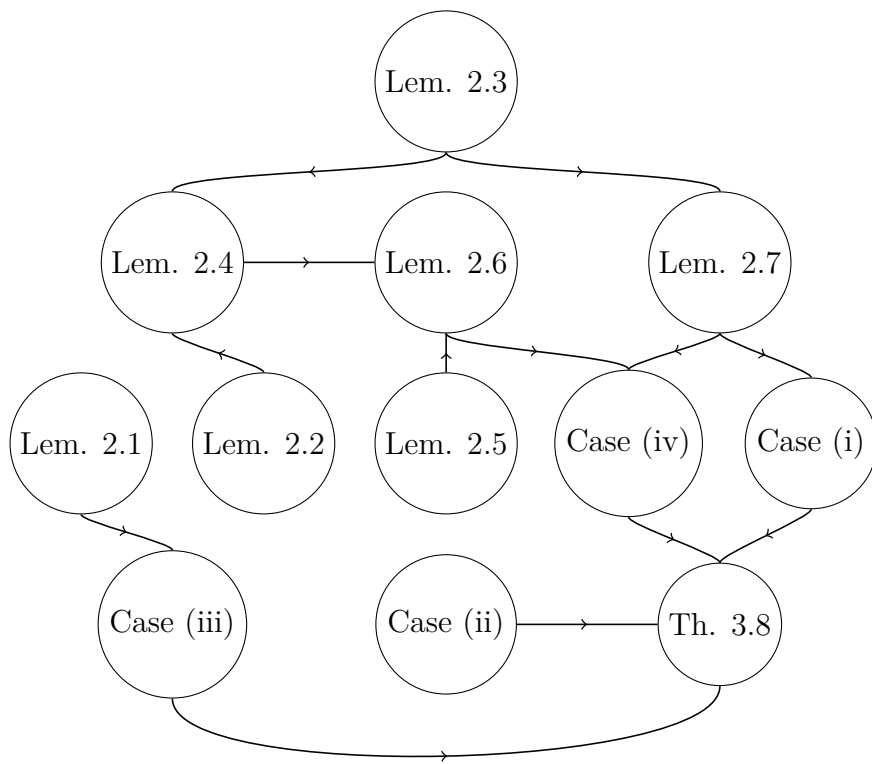
\begin{center} 
\begin{tikzpicture}
[scale=0.8]
 \tikzset{
    ell/.style={circle,draw,minimum height=0.2cm,minimum width=0.3cm,inner sep=0.1cm},
    arrowmark/.style={
        postaction={decorate,decoration={markings,mark=at position 0.5 with {\arrow{>}}}}
    }
}
\node[ell] (l1) at (-2,2) {Lem. 2.1};
\node[ell] (l2) at (1,2) {Lem. 2.2};
\node[ell] (l3) at (4,8) {Lem. 2.3};
\node[ell] (l4)at (-0.5,5) {Lem. 2.4};
\node[ell] (l5) at (4,2) {Lem. 2.5};
\node[ell] (l6) at (4,5) {Lem. 2.6}; 
\node[ell] (l7) at (8.5,5) {Lem. 2.7};
\node[ell] (c1) at (10,2) {Case (i)};
\node[ell] (c2) at (4,-1) {Case (ii)};
\node[ell] (c3)at (-0.5,-1) {Case (iii)};
 \node[ell] (c4) at (7,2) {Case (iv)};
\node[ell] (thm) at (8.5,-1) {Th. 3.8};
\draw  [arrowmark, line width=0.6pt]  (l1) to  [out=270,in=90,looseness=0.2]  (c3);
\draw [arrowmark, line width=0.6pt] (l2) to  [out=90,in=270,looseness=0.2]  (l4);
\draw [arrowmark, line width=0.6pt](l3) to  [out=270,in=90,looseness=0.2]  (l4);
\draw [arrowmark, line width=0.6pt] (l3) to  [out=270,in=90,looseness=0.2]  (l7);
\draw [arrowmark, line width=0.6pt] (l4) to  [out=0,in=180,looseness=0.2]  (l6);
\draw [arrowmark, line width=0.6pt] (l5) to  [out=90,in=270,looseness=0.2]  (l6);
\draw [arrowmark, line width=0.6pt] (l6) to  [out=270,in=90,looseness=0.2]  (c4);
\draw [arrowmark, line width=0.6pt] (l7) to  [out=270,in=90,looseness=0.2]  (c4);
\draw [arrowmark, line width=0.6pt] (l7) to  [out= 270,in=90,looseness=0.2]  (c1);
\draw [arrowmark, line width=0.6pt] (c1) to  [out=270,in=90,looseness=0.2]  (thm);
\draw [arrowmark, line width=0.6pt] (c2) to  [out=0,in=180,looseness=0.2]  (thm);
\draw [arrowmark, line width=0.6pt] (c3) to  
 [out=270,in=270,looseness=0.4] 
 (thm);
\draw [arrowmark, line width=0.6pt] (c4) to  [out=270,in=90,looseness=0.2]  (thm);
\end{tikzpicture}
\end{center}
while the following lemmas are  independent and rely on sieve theory:
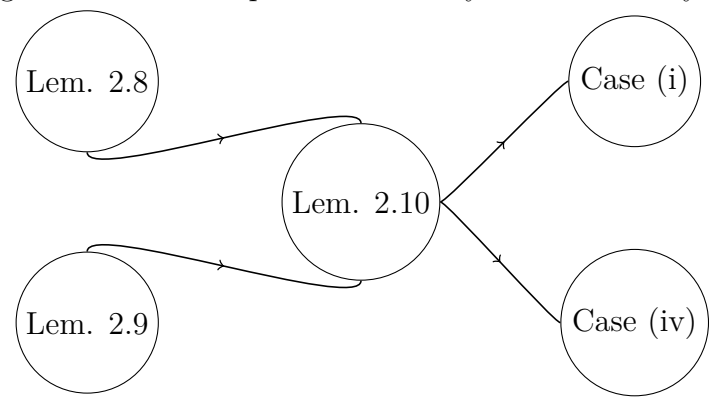
\begin{center} 
\begin{tikzpicture}
[scale=0.8]
\tikzset{
    ell/.style={circle,draw,minimum height=0.2cm,minimum width=0.3cm,inner sep=0.1cm},
    arrowmark/.style={
        postaction={decorate,decoration={markings,mark=at position 0.5 with {\arrow{>}}}}
    }
}
 \node[ell] (l4)at (-0.5,4) {Lem. 2.8};  
\node[ell] (l5) at (4,2) {Lem. 2.10};   
\node[ell] (l7) at (8.5,4) {Case (i)};
 \node[ell] (c3)at (-0.5,0) {Lem. 2.9};  
 \node[ell] (thm) at (8.5,0) {Case (iv)};  
\draw [arrowmark, line width=0.6pt] (l4) to  [out=270,in=90,looseness=0.3]  (l5);
\draw [arrowmark, line width=0.6pt] (c3) to  [out=90,in=270,looseness=0.3]  (l5);
\draw [arrowmark, line width=0.6pt] (l5) to  [out=0,in=180,looseness=0.1]  (l7);
\draw [arrowmark, line width=0.6pt] (l5) to  [out=0,in=180,looseness=0.1]  (thm);
 \end{tikzpicture}
\end{center}

 The work of~\cite[Lemma~1]{MR552470}  gives an upper bound on the density  of integers all of whose prime factors are relatively small.
We shall need a variation of this result where the integers are weighted by a multiplicative function. In the applications it will be important that the bound 
is of the form $O(x^{o(1)}z^{-c})$ for some positive constant $c$.

\begin{lemma}
\label{lemshiu} 
Fix any positive real numbers $c_0,c_1,c_2$ and assume that $F:\mathbb N\to [0,\infty)$ is a multiplicative function such that 
\begin{align}
\label{eFbound}
F(p^e) \leq \min \left\{\frac{c_0}{p}, \frac{p^{c_1}}{p^{ec_2}}\right\}
\end{align}
for all primes $p$ and $e\geq 1 $. Define 
$$
c:= \min\left\{ \frac{c_2}{2} , \frac{1}{1+[2c_1/c_2]}\right\}
\ \textrm{ and } \ c':=\frac{c+ 2(c_0+c) }{c}.
$$ 
Then for all $x, z\geq 2$ we have  
$$
\sum_{\substack{ n\in \mathbb N \cap (z,x] \\  p\mid n \Rightarrow p\leq (\log x)(\log \log x)} } F(n) 
\ll 
z^{-c}\exp\left(  \frac{c'\log x}{(\log \log x)^{1/2}}   \right)
,$$ where the implied constant is absolute.
\end{lemma}

\begin{proof}  
Let $c_4$ be a positive constant that will be optimised later. Then the sum over $n$ is 
$$
\leq \frac{1 }{z^{c_4} } 
\sum_{\substack{ n \leq x  \\  p\mid n  \Rightarrow   p\leq y } } F(n ) n^{c_4},
$$ 
where $y=(\log x)(\log \log x)$. By Rankin's trick we get the following   bound for any $\delta>0$: 
$$
\leq \frac{x^\delta}{z^{c_4} } 
\sum_{\substack{ n\in \mathbb N \\  p\mid n \Rightarrow   p\leq y} } F(n) \frac{n^{c_4}}{n^\delta }
=\frac{x^\delta}{z^{c_4} }   \prod_{p\leq y} \left(1+\sum_{e\geq 1 } F(p^e) p^{e(c_4-\delta)} \right).
$$ 
For an auxiliary positive integer $e_0$ we shall control the contribution of the 
range $ e\leq  e_0$ and $e> e_0$   using the bounds $F(p^e) \leq c_0/p$ and 
  $F(p^e) \leq p^{c_1-ec_2}$ respectively. Assume that 
$c_4 \geq \delta $ so that the contribution of the former  range contributes  
$$\leq 1+\sum_{e=1}^{e_0} \frac{c_0}{p} \frac{p^{ec_4}}{p^{e\delta}} 
\leq 1+c_0 e_0 p^{e_0(c_4-\delta) -1} .$$ Now assume that $c_4 e_0 \leq 1 $ 
so that the bound becomes 
$$\leq 1+\frac{c_0}{c_4}  p^{ -\delta e_0} \leq 1+\frac{c_0}{c_4}  p^{ -\delta  }\leq 1+\frac{c_0}{c_4} \frac{1}{p^\delta -1 } .$$ 
The remaining range contributes 
 $$\leq p^{c_1 }
\sum_{e\geq 1+ e_0  } p^{e(c_4-\delta-c_2)}  
.$$ Making the additional assumption that $c_4\leq  \frac{1}{2} c_2 $ we can bound this by 
 $$\leq p^{c_1 }
\sum_{e\geq 1+ e_0  } p^{- e (\delta+c_2/2)}
\leq  \frac{p^{c_1 }}{p^{e_0(\delta+c_2/2)} }\sum_{j=1}^\infty \frac{1}{p^{j(\delta+c_2/2)}}
\leq  \frac{p^{c_1 }}{p^{e_0 c_2/2} } \frac{1}{p^\delta -1 }
.$$ Further assuming that $2c_1\leq e_0 c_2$ shows that this is $\leq \frac{1}{p^\delta -1 }$.
Putting the bounds together leads to    
$$ 
1+\sum_{e\geq 1 } F(p^e) p^{e(c_4-\delta)} \leq 1+\frac{c_0+c_4 }{c_4} \frac{1}{p^\delta -1 }  , 
$$ 
subject to the conditions 
$$  
\delta\leq c_4,
c_4 e_0 \leq 1 ,
c_4\leq \frac{c_2}{2} ,
\frac{ 2c_1}{c_2 } \leq e_0.
$$ 
Putting $e_0= 1+[2c_1/c_2]$ shows that 
these conditions are met for any $\delta \in (0,c_4)$ where   
$$ 
c_4:= \min\left\{ \frac{c_2}{2} , \frac{1}{1+[2c_1/c_2]}\right\} = c.
$$
Hence, the overall bound becomes 
$$\frac{x^\delta}{z^{c_4} }   \prod_{p\leq y} \left(
1+\frac{c_0+c_4 }{c_4} \frac{1}{p^\delta -1 } 
 \right)\leq \frac{x^\delta}{z^{c_4} }   \prod_{p\leq y} 
\left( 1+ \frac{1}{p^\delta -1 }\right) ^{\hspace{-0,1cm}\frac{c_0+c_4 }{c_4}} \hspace{-0,3cm}
\leq z^{-c_4}\exp\left( \delta (\log x)+\frac{c_0+c_4 }{c_4} T\right), $$ 
where 
$$T:= \sum_{p\leq y } \frac{1}{p^\delta -1 } \leq \frac{1}{\delta}  \sum_{p\leq y }\frac{1}{\log p } \leq \frac{1}{\delta} \frac{2 y }{(\log y)^2} 
 ,$$ owing to the inequality  $\mathrm e^t -1 \geq t$ with $t=\delta \log p$ and the prime number theorem.
Putting $\delta=(\log \log x)^{-1/2}$ we see that when $y= (\log x) (\log \log x)$ one gets the bound 
$$z^{-c_4}\exp\left( \frac{\log x}{(\log \log x)^{1/2}}+\frac{(c_0+c_4) }{c_4} (\log \log x)^{1/2} \frac{2 (\log x) (\log \log x) }{(\log \log x)^2}  \right),
$$ which is sufficient. \end{proof}

\begin{lemma}
\label{lemshiu2} 
Keep the setting of Lemma~\ref{lemshiu}
and fix any $\beta_0>0$. For all  $T\geq 2$ with $\log T>4\beta_0/c_2$, for all $A>1, c\in \mathbb N$, and for any $\beta>0$ with 
$$
\beta \leq  \min\left\{ \frac{c_2}{2 },  
\frac{\beta_0}{\log T}   
\right\},
$$  
the product   
$$
\prod_{\substack{ p\leq T \\ p\nmid c }} \Big(1+   \sum_{\substack{ i\geq 1  \\ j\geq 0  }} 
 \min\{C'p^{ (i+j) c_2/2},A ^{i+j } \} 
F(p^{i+j}) (p^{\beta i }-p^{\beta(i-1)})
\Big(\frac{\mathds 1 [p>c_0]}{ (1-F(p ) ) }  + \mathds 1 [p\leq c_0] \Big) 
\Big)$$ is $ O( \mathrm e^{\nu    \beta \log T  })$,   where $\nu$ is a positive constant that depends at most on $\beta_0$, $c_i$ and $A$. Furthermore, the  implied constant depends at most on  
$A,C', c_0, c_1$ and $c_2$.
\end{lemma}

\begin{proof} 
Define  $p_0$ to be the least prime satisfying $2 A    \leq  p_0^{c_2/4 }$. We will bound the sum over $i, j $ for every individual
 prime $p\geq p_0$ and in the end we shall piece the bounds together for all primes $p\leq T$.
\\
 \textbf{Step (1).}
We   start with the contribution  of large $i$, in which case the  bound 
$F(p^e) \leq  p^{c_1- e c_2} $ and the    crude estimate 
 $ p^{\beta i }-p^{\beta(i-1)} \leq  p^{\beta i } $ will suffice.
Define 
$$ i_1 :=1+\left[\frac{4(5+c_1)}{c_2}\right] .$$ 
 The contribution of $i\geq i_1 $ is 
$$
\leq 
\sum_{ i\geq i_1  } A^{i}  p^{\beta i }  
\sum_{   j\geq 0  }  A^{j} F(p^{i+j})    
\leq p^{c_1}
\sum_{ i\geq i_1  } A^{i}  p^{(\beta -c_2) i }  
\sum_{   j\geq 0  }  (Ap^{-c_2 })^{j }  \leq 2 p^{c_1}
\sum_{ i\geq i_1  } (A  p^{(\beta -c_2)   }  )^i
 $$ because  $Ap^{-c_2 }\leq 1/2$, a fact that follows from 
  $p\geq p_0$. Now we use the assumptions 
 $ \beta \leq  c_2/2$ and  $2A      \leq  p_0^{   c_2/4  }  \leq  p^{   c_2/4  }$  to see that  
$A  p^{(\beta -c_2)   } \leq A p^{  -c_2/2   }  \leq  1/2 $.
Hence,  
$$
2 p^{c_1}
\sum_{ i\geq i_1  } (A  p^{(\beta -c_2)   }  )^i 
\leq  
4 p^{c_1}  (A  p^{(\beta -c_2)   }  )^{i_1}  
\leq 4 p^{c_1-i_1 c_2/4}  \leq    p^{2+c_1-i_1 c_2/4} .$$ This is $\leq p^{-3}$  because our choice for $i_1$ makes sure that 
 $5+c_1\leq i_1 c_2/4.$ We have thus shown that for all $p\geq p_0$ one has  
$$ \sum_{\substack{ i\geq i_1  \\ j\geq 0  }} A^{i+j} F(p^{i+j}) (p^{\beta i } - p^{\beta(i-1) }) \leq p^{-3}.$$
\textbf{Step (2).} Let us now bound the contribution of the $i,j$ that satisfy 
$$1\leq i < i_1 \ \ \textrm{ and } \ \ i+j \geq i_1.
$$
We have 
 $$ \sum_{ i=1}^{i_1-1}   \sum_{ j \geq i_1-i}  A^{i+j} F(p^{i+j}) (p^{\beta i } - p^{\beta(i-1) })
\leq p^{c_1}
\sum_{ i=1}^{i_1-1}    A^i p^{-i c_2}  (p^{\beta i } - p^{\beta(i-1) })
\sum_{ j \geq i_1-i}  (A p^{ - c_2})^j 
.$$ Using   the inequality  $A p^{ - c_2}\leq 1/2$ to bound  the sum over $j$  results in the inequality  
  $$\leq  2 p^{c_1} \sum_{ i=1}^{i_1-1}    A^i p^{-i c_2}  (p^{\beta i } - p^{\beta(i-1) })
   (A p^{ - c_2})^{i_1-i} =2 p^{c_1} A^{i_1} p^{-c_2 {i_1} }
\sum_{ i=1}^{i_1-1}    (p^{\beta i } - p^{\beta(i-1) })
,$$ which is at most $  2 p^{ c_1} ( A  p^{ ( \beta -c_2) } )^{ i_1} $ 
that has been previously shown to be   at most  $\leq 1/p^3$.

We have thus proved  that for all $p\geq p_0$ one has  
$$ \sum_{\substack{ 1\leq i < i_1  \\ j\geq i_1-i   }} A^{i+j} F(p^{i+j}) (p^{\beta i } - p^{\beta(i-1) }) \leq p^{-3}.$$
 \textbf{Step (3).}
It remains to study the contribution of cases with 
 $  i+j < i_1 $.  For these we use the assumption $F(p^e) \leq c_0/p$ that leads to the bound  
 $$
\leq \frac{c_0}{p}  \sum_{\substack{ i\geq 1, j\geq 0  \\ i+j < i_1  }} A^{i+j}   (p^{\beta i } - p^{\beta(i-1) })
\leq  \frac{c_0}{p} 
 \sum_{1\leq i <i_1 } (2A)^{i}  (p^{\beta i } - p^{\beta(i-1) })
 \sum_{\substack{ 0 \leq  j < i_1-i   }} (2A)^{j}   
.$$ Now since $A>1$ we have  $2A>2$. For all   $m\geq 1 $ we have 
$$
1+(2A)+(2A)^2+\ldots+(2A)^{m-1} \leq  \frac{(2A)^m}{2A-1} \leq (2A)^m.
$$ 
This gives the bound $$\leq  \frac{c_0}{p}  \sum_{1\leq i <i_1 } (2A)^{i}  (p^{\beta i } - p^{\beta(i-1) })
 (2A)^{ i_1-i } = \frac{c_0(2A)^{ i_1  } }{p}   \sum_{1\leq i <i_1 }  (p^{\beta i  } - p^{\beta(i-1) })
\leq \frac{c_0 (2A)^{ i_1  } }{p}   (p^{\beta i_1 } - 1)
.$$

The assumption  
$\beta \log T \leq \beta_0$ 
shows that $\beta i_1 \log T \leq \beta_0(1+ \frac{4(5+c_1)}{c_2})$, hence, for primes  $p\leq T$ we infer  $ \beta i_1 \log p \leq \beta_0(1+ \frac{4(5+c_1)}{c_2})  $.
On the other hand, the function   $ (-1+\mathrm e ^t)/t $ is bounded in the interval 
   $0\leq t \leq \beta_0(1 + \frac{4(5+c_1)}{c_2})   $, thus, 
  $$p^{\beta  i_1} - 1 =\exp ((\log p ) \beta i_1) -1 \leq \beta_1 (\log p ) \beta i_1 ,$$ for a positive 
constant $ \beta_1$ that depends on $\beta_0$ and $ c_1,c_2$. Thus, 
 the contribution of cases with  $  i+j < i_1 $ is  
$$
\leq \frac{c_0 (2A)^{ i_1  } }{p}   (p^{\beta i_1 } - 1)
\leq \left\{  c_0 (2A)^{ i_1  }  \beta_1  \beta i_1 \r\} \frac{\log p }{p}.
$$ 
In conclusion, we saw    that for all primes $p \in (p_0, T]$ one has 
 $$
\sum_{\substack{ i\geq 1  \\ j\geq 0  }} A^{i+j} F(p^{i+j}) (p^{\beta i }-p^{\beta (i-1)  })
\leq 2p^{-3} +\left\{  c_0 (2A)^{ i_1  }  \beta_1  \beta i_1 \r\} \frac{\log p }{p}
.$$\textbf{Step (4).} Using the last inequality   with the bound 
$ 1+x_p\leq \exp(  x_p)$, valid for all $x_p \in \mathbb R$,
shows that, once restricted in the range $p>\max\{p_0,c_0\}$, 
  the product in the lemma is  
$$
\leq \exp\Big(\sum_{\substack{\max\{p_0 , c_0\} 
< p\leq T  \\ p\nmid c} }  
(1-F(p ) ) ^{-1} 
\Big( 
2p^{-3} +\left\{  c_0 (2A)^{ i_1  }  \beta_1  \beta i_1 \r\} \frac{\log p }{p}\Big)
 \Big) 
.$$  Ignoring the condition $p\nmid c $ will produce a larger bound. 
Using the inequality  $F(p) \leq c_0/p$ 
we obtain  $$
\ll\exp\Big(
c_0 (2A)^{ i_1  }  \beta_1  \beta i_1
\sum_{\substack{\max\{p_0 , c_0\}  < p\leq T  } }   (1-F(p ) ) ^{-1}  \frac{ \log p    }{p}  
 \Big) 
,$$ where the implied constant depends at most on $c_0$. Using the inequality  $F(p) \leq c_0/p$   
and the estimate $\sum_{p\leq y }(\log p )/p \ll \log y $ leads to  
 $$\sum_{\substack{\max\{p_0 , c_0\}  < p\leq T  } }    \frac{ \log p    }{  (1-F(p ) )p}  
\leq \sum_{\substack{\max\{p_0 , c_0\}  < p\leq T  } }  \frac{ \log p    }{  p}   \left(1+O_{c_0}\left(\frac{c_0}{p}\right)\right)
\ll O_{c_0}(1)+ \log T  , $$  
where the implied constant is absolute. Hence, the previous bound becomes 
$$\ll_{c_0} \exp\Big( c_0 (2A)^{ i_1  }  \beta_1  \beta i_1   \log T     \Big) .$$ Since $i_1$ is a function of $c_1$ and $c_2$ 
we can thus write the bound as   $\ll_{ c_0} \exp (  \nu   \beta \log T   ) $ for some $ \nu=\nu (\beta_0,c_0,c_1,c_2,A) $.  
To conclude the proof of the lemma 
we must deal with the    contribution of the primes  $p\leq \max\{ p_0,c_0\}$. Note that for every prime $p$ the corresponding factor in the product of the lemma is 
$$\leq 1+    \Big( \frac{1}{ 1-F(p )  }  + 1 \Big) \sum_{\substack{ i\geq 1  \\ j\geq 0  }}   \min\{C'p^{ (i+j) c_2/2},A ^{i+j } \} F(p^{i+j}) 
p^{\beta i } .$$ Using the bound for $F$ in the assumptions of Lemma~\ref{lemshiu} and the bound $\beta \leq \beta_0/\log T$ we see that 
the sum over $i, j $ is at most  $$C' p^{c_1}\sum_{i\geq 1 }   p^{   ( -c_2/2+\beta_0/\log T) i  }  \sum_{  j\geq 0  } p^{ - j c_2/2}  .$$ 
Our assumption $4\beta_0/c_2  < \log T$ ensures that  $\beta_0/\log T < c_2/4$, hence, we obtain the bound 
 $$C' p^{c_1}\sum_{i\geq 1 }   p^{    -i c_2/4   }  \sum_{  j\geq 0  } p^{ - j c_2/2} \leq C' p^{c_1}\sum_{i\geq 1 }   2^{    -i c_2/4   }  \sum_{  j\geq 0  } 2^{ - j c_2/2} =O_{c_2}(C' p^{c_1})  .$$ Taking the product of this quantity over all primes $p\leq \max\{ p_0,c_0\}$ gives an implied constant 
that depends on $p_0,c_0,c_1,c_2$ and $C'$. Since  $p_0=p_0(A,c_2)$ we see that the implied constant also depends on $A$.
 \end{proof}

\begin{lemma} \label{lem:postvenice}Fix any positive constants $C, C', \epsilon$ and assume that we are given a function $G:\N \to [0,\infty)$
such that for all coprime positive integers $a,b$ one has  $$G(ab ) \leq  G(a) \min\{C^{\Omega(b)}, C'b^{\epsilon}\}.$$
Then for all coprime positive integers $a,b$ we have  $G(ab ) \leq  G(a) H(b)$, where $H$ is the multiplicative function  
defined as $H(p^e)= \min\{C^e, C'p^{\epsilon e}\} $ for all $e\geq 1$ and primes $p$.
\end{lemma}\begin{proof} We will prove this with induction on $\omega(b)$. When $\omega(b) =0 $ then $b=1$, hence,  the statement clearly holds. 
Assume that $k\geq 0 $ and that the statement holds for all $b\in \N$ with $\omega(b)=k$. 
Now let  $n, n' $ be coprime and assume that $\omega(n)=k+1$. We shall show that $G(n'n)\leq G(n') H(n)$.
Writing  $n =p_1^{\alpha_1} \cdots p_k^{\alpha_k} p_{k+1}^{\alpha_{k+1} }$ where each $\alpha_i$ is strictly positive and the $p_i$ are 
distinct primes,  we let $a= n' p_1^{\alpha_1} \cdots p_k^{\alpha_k} $ and $b=p_{k+1}^{\alpha_{k+1} }$ so that  
$$ G(n'n ) =G(ab) \leq G(a) H(p_{k+1}^{\alpha_{k+1} })$$ by assumption. Now $a$ can be written as $n'$ multiplied by an integer that is coprime to $n'$
and with exactly $k$ distinct prime factors, thus,       our inductive hypothesis 
shows that $$ G(a)\leq G(n') \prod_{i=1}^k H(p_i^{\alpha_i} ).$$ Combining the two inequalities  gives 
 $ G(n'n ) \leq G(n') \prod_{i\leq k+1} H(p_i^{\alpha_i} )=G(n') H(n)$.\end{proof}

\begin{lemma}\label{lem:bwv1046}Keep the setting of Lemma~\ref{lemshiu}, fix any $C>1$
 and $C'>0$
 and assume that $G:\mathbb N \to [0,\infty)$ is a function that satisfies  
\begin{align}
\label{eGbound}
G(ab) \leq G(a) \min\{C^{\Omega(b)}, C' b^{c_2/2}\}
\end{align}
 for all coprime positive integers $a,b$.
Fix any positive real number $\beta_0$. For any $\Upsilon   , \Psi \geq  2  $ and $\varpi >0$ satisfying 
$$ \varpi 
 \leq  \min\left\{\frac{c_2}{2 }\log \Psi ,  \beta_0   \right\}
$$
we have  
 $$ \sum_{\substack{ a>\Upsilon   \\ P^+(a)<\Psi }}    F(a)G(a)    \prod_{\substack{ c_0<p\\ p\mid a  } } (1-F(p ) ) ^{-1} 
 \ll    \exp \left(  - \varpi   \frac{\log \Upsilon }{\log \Psi} 
  \right)   \sum_{\substack{   n   \in \mathbb N \\ P^+( n  )< \Psi   }}     F(n ) G(n)  \prod_{\substack{ c_0<p \\  p\mid   n    } } (1-F(p ) ) ^{-1} 
,$$  where  the implied constant depends at most on $C, C', \beta_0$ and $c_i$.  \end{lemma} 
 
\begin{proof} 
Define $ \beta  := \varpi /\log \Psi $. The sum is at most 
$$ \sum_{  P^+(a)<\Psi }      F(a) G(a) \prod_{\substack{ c_0<p\\ p\mid a  } } (1-F(p ) ) ^{-1} 
\left(\frac{a}{\Upsilon }\right)^\beta .$$  Now define the multiplicative function $\psi_\beta:\mathbb N\to \mathbb R$ via the Dirichet convolution 
$$m^\beta=\sum_{\substack{ d\in \mathbb N \\ d\mid m} }\psi_\beta(d), \ \ \ m \in \mathbb N.$$ Writing $n=a/d$ we obtain   $$  \Upsilon^{-\beta} 
\sum_{\substack{ d \in \mathbb N \\ P^+(d)< \Psi }}  \psi_\beta(d)
\sum_{\substack{   n \in \mathbb N \\ P^+(n)<\Psi   }} 
F(nd)      G(nd) \prod_{\substack{ c_0<p\\ p\mid nd  } } (1-F(p ) ) ^{-1}  .$$
Now factor $n=n_0 n_1 $, where $\gcd(n_1,d)=1$ and $n_0$ is only divisible by primes dividing $d$. Then the sum over $d$ and $n$ becomes 
$$ \sum_{\substack{ d \in \mathbb N \\ P^+(d)<\Psi  }}  \psi_\beta(d)
\sum_{\substack{ \mathbf n   \in \mathbb N ^2, \ P^+(n_0 n_1 )<\Psi  \\ p\mid n_0 \Rightarrow p\mid d \\ \gcd(n_1,d)=1 }} 
F(n_0 n_1 d)      G(n_0 n_1 d) \prod_{\substack{ c_0<p \\  p\mid n_0 n_1 d  } } (1-F(p ) ) ^{-1} 
.$$  Our assumptions on $G$   together with Lemma~\ref{lem:postvenice} ensure that    $G(n_0 n_1 d)     \leq G(n_1) H(n_0 d) $, where $H$ is the multiplicative function given by $ H(p^e)= \min\{C ^e, C'p^{ e c_2/2}\}$ for  $e\geq 1$ and primes $p$. Together with the multiplicativity of $F$ we obtain the   bound 
\begin{align*} 
\Upsilon^{-\beta}   
&
\sum_{\substack{   n_1  \in \mathbb N \\ P^+( n_1 )<\Psi    }}  F(n_1)  G(n_1)
\prod_{\substack{ c_0<p \\  p\mid   n_1    } } (1-F(p ) ) ^{-1} 
\\  \times & \sum_{\substack{   n_0, d    \in \mathbb N,  P^+(d  )<\Psi  \\ p\mid n_0 \Rightarrow p\mid d \\ \gcd(d, n_1)=1 }} 
F(n_0 d) 
 H(n_0 d)    \psi_\beta(d) \prod_{\substack{ c_0<p \\  p\mid     d  } } (1-F(p ) ) ^{-1} 
.\end{align*}  It is easy to see that  $ \psi_\beta( p^m ) = p^{\beta m } - p^{\beta(m-1) } $ for all $m\geq 1 $ and primes $p$.
We can use this to write  the sum over $n_0, d $    as an Euler product.
The Euler product is of the type     covered by   Lemma~\ref{lemshiu2} as can be seen by taking
$A=C, c=n_1$ and  $T=\Psi$. The assumption of the present lemma on the size of $\varpi$ implies that the assumption of  Lemma~\ref{lemshiu2} on the size of $\beta$.
Thus,   the sum over $a$ in the lemma is 
$$   \ll  \frac{  \mathrm e^{\nu \beta \log \Psi }}{\Upsilon^{ \beta} }
\sum_{\substack{   n   \in \mathbb N \\ P^+( n  )< \Psi   }}     F(n ) G(n)  \prod_{\substack{ c_0<p \\  p\mid   n    } } (1-F(p ) ) ^{-1} 
,$$  where $\nu=\nu(\beta_0, c_0,c_1,c_2, C)$ is positive and 
the implied constant depends at most on $C$ and $c_i$. 
Using the fact that $\varpi \leq \beta_0$,
we can write 
$$  \mathrm e^{\nu \beta \log \Psi }
= \mathrm e^{\nu  \varpi }=O_{\beta_0}(1).$$ Finally, 
we have $ \Upsilon^{ -\beta} = \exp(   - \varpi    \log \Upsilon /\log \Psi) $.
 \end{proof}

\begin{lemma}\label{lem:brukner8} Keep the setting of Lemma~\ref{lem:bwv1046} and define for any $V\geq 1 $ the function  $$\mathcal H (V):=  \sum_{\substack{   n   \in \mathbb N \\ P^+( n  )< V   }}  F(n )  G(n ) \prod_{\substack{ c_0<p \\  p\mid   n    } } (1-F(p ) ) ^{-1}   .$$
  For $V\geq 1 $ and $\epsilon>0$ with 
$V^{\epsilon c_2/2} >2C$ and $V^\epsilon>c_0$ 
we have 
$$
\mathcal H (V)
\ll \frac{ \mathcal H (V^\epsilon)}{\epsilon^{\nu_1}}
,$$ where $\nu_1=\nu_1(C, c_0,c_1,c_2)$ is positive and 
the implied constant depends at most on $C$, $C'$ and $c_i$.  \end{lemma} \begin{proof} 
For a prime $p>V^\epsilon$ we have  $p^{c_2/2} >2C $ due to the assumption   $V^{\epsilon c_2/2} >2C$.
Now let $j_0 :=1+[ 4/ c_2 +2c_1/ c_2] $ so that $j_0 c_2  \geq 4+2 c_1  $. Then 
   $- j_0c_2 /2 \leq  - c_1-2  $, which can be combined with $p^{c_2/2} >2C$ 
to show that $$
(   C p^{-c_2})^{j_0} 
\leq 
   p^{-j_0 c_2/2}  \leq  p^{-c_1-2} 
 . $$By~\eqref{eFbound} and the fact that $C>1$ we see that  
$$
\sum_{j=1}^\infty C^{j } F(p^j) 
\leq j_0 C^{j_0}  \frac{c_0}{p} + p^{c_1 }\sum_{j=1+j_0 }^\infty ( C  p^{- c_2})^{j}
\leq j_0 C^{j_0}  \frac{c_0}{p} + p^{c_1 } 2( C  p^{- c_2})^{j_0}
\leq j_0 C^{j_0}  \frac{c_0}{p} + \frac{2}{p^2} , $$ 
which is at most $\nu_1/p$, where $\nu_1$ is a positive constant that depends at most on $  C$ and $c_i$.
We infer that 
\begin{equation}\label{eq:montspirit}
\hspace{-0,4cm}
\prod_{p\in ( V^\epsilon, V) } \!\left( 1+ \sum_{j=1}^\infty 
\frac{ C^{j } F(p^j) }{1-F(p) } \right)\leq 
\prod_{p\in ( V^\epsilon, V)} \!\left(1+ \frac{ \nu_1/p  }{1-c_0/p } \right)\ll 
\prod_{p\in ( V^\epsilon, V)} \!\left( 1+ \frac{1}{p} \right)^{\nu_1 } \!\!\ll \frac{1}{\epsilon^{\nu_1}}
,\end{equation} with an implied constant that depends at most on  $C$ and $c_i$.

We can now use~\eqref{eq:montspirit} to bound  $\mathcal H (V)$.
Each positive integer $n$ can be written uniquely as $n=n_0 n_1$, where $P^+(n_0)\leq V^\epsilon$ and $P^-(n_1)> V^\epsilon$. 
We have $G(n_0 n_1   )    \leq G(n_0) C^{\Omega(n_1 )}$
by equation~\eqref{eGbound} and together with the 
multiplicativity of $F$ we obtain 
$$   \mathcal H (V) \leq 
\sum_{\substack{   n_0   \in \mathbb N \\ P^+( n_0  )\leq  V^\epsilon   }}  F(n_0 )  G(n_0 ) \prod_{\substack{ c_0<p \\  p\mid   n_0    } } (1-F(p ) ) ^{-1} 
\sum_{\substack{   n_1   \in \mathbb N \\ P^-( n_1  )>  V^\epsilon \\ 
P^+( n_1  ) < V}}  C^{\Omega(n_1 )}  F(n_1 ) \prod_{\substack{ c_0<p \\  p\mid   n_1    } } (1-F(p ) ) ^{-1} 
.$$ The assumption  $V^\epsilon>c_0$ shows that every prime $p>V^\epsilon$ satisfies
$p>c_0$, hence, the sum over $n_1$ equals $$\prod_{ V^\epsilon < p < V} \left( 1+ (1-F(p ) ) ^{-1} \sum_{j=1}^\infty C^{j } F(p^j) \right).$$
Alluding to~\eqref{eq:montspirit} and noting that the sum over $n_0$ equals $\mathcal H(V^\epsilon)$ concludes the proof. 
\end{proof}

\begin{lemma}
\label{lem:final}
Let $F$ be as in Lemma~\ref{lemshiu} and $G$ be as in Lemma~\ref{lem:bwv1046}. 
Fix any positive real number $\beta_0$. 
For sufficiently large $\Upsilon   , \Psi \geq  2  $ and for all $\varpi >0$ satisfying 
$$ \varpi 
 \leq  \min\left\{\frac{c_2}{2 }\log \Psi ,  \beta_0   \right\}
$$  we have  
$$ \sum_{\substack{ a>\Upsilon   \\ P^+(a)<\Psi }} F(a)  
G(a) \prod_{\substack{ c_0<p\\ p\mid a  } } (1-F(p ) ) ^{-1} 
 \ll    \exp \left(   -   \varpi \frac{\log \Upsilon }{\log \Psi}  
  \right)   \sum_{\substack{ a\leq \Psi   }} F(a)     G(a) \prod_{\substack{ c_0<p\\ p\mid a  } } (1-F(p ) ) ^{-1}
,$$ where the implied constant depends at most on   $C, \beta_0$ and $c_i$. \end{lemma} 

\begin{proof} 
Taking $\Psi=\Upsilon^\epsilon$ and $\beta_0  =\varpi=1  $ 
 in Lemma~\ref{lem:bwv1046} shows that  
$$ \sum_{\substack{ a>\Upsilon   \\ P^+(a)<\Upsilon^\epsilon }} F(a)  
G(a) \prod_{\substack{ c_0<p\\ p\mid a  } } (1-F(p ) ) ^{-1} 
 \ll      \frac{  \mathcal H(\Upsilon^\epsilon )  }{ \exp(  1  /\epsilon)}
,$$  where   the implied constant depends at most on $C$, $C'$ and $c_i$. Taking a sufficiently small $\epsilon=\epsilon_0$ in terms of 
$C$ and $c_i$ makes the right-hand side be $\leq  \mathcal H(\Upsilon^{\epsilon_0} ) /2$.
Furthermore, by the definition of $\mathcal H$ we have 
$$  \mathcal H(\Upsilon^{\epsilon_0} ) \leq  \sum_{\substack{ a\leq \Upsilon   \\ P^+(a)<\Upsilon^{\epsilon_0} }} F(a)  
G(a) \prod_{\substack{ c_0<p\\ p\mid a  } } (1-F(p ) ) ^{-1} 
+ \sum_{\substack{ a>\Upsilon   \\ P^+(a)<\Upsilon^{\epsilon_0} }} F(a)  
G(a) \prod_{\substack{ c_0<p\\ p\mid a  } } (1-F(p ) ) ^{-1} 
, $$thus,  $$\mathcal H(\Upsilon^{\epsilon_0} ) \leq  \sum_{\substack{ a\leq \Upsilon   }} F(a)  
G(a) \prod_{\substack{ c_0<p\\ p\mid a  } } (1-F(p ) ) ^{-1} +\frac{\mathcal H(\Upsilon^{\epsilon_0} )}{2}
.$$ Hence,  $$\mathcal H(\Upsilon^{\epsilon_0} ) \leq 2 \sum_{\substack{ a\leq \Upsilon   }} F(a)  
   G(a) \prod_{\substack{ c_0 <p\\ p\mid a  } } (1-F(p ) ) ^{-1}  .$$ Thus,   by Lemma~\ref{lem:brukner8} 
we infer that $$  \mathcal H(\Psi  )  \ll  \frac{\mathcal H(\Psi ^{\epsilon_0} )   }{\epsilon_0^{\nu_1 }}\ll
\sum_{\substack{ a\leq \Psi   }} F(a)     G(a) \prod_{\substack{ c_0<p\\ p\mid a  } } (1-F(p ) ) ^{-1} .$$
We conclude the proof by  injecting this estimate  into  Lemma~\ref{lem:bwv1046}.
 \end{proof}

\begin{lemma} 
\label{lem:damnnotfinal}
Let $F$ be as in Lemma~\ref{lemshiu} and $G$ be as in Lemma~\ref{lem:bwv1046}. 
Fix any positive constant $\gamma$ and assume that for every prime $p$ we are given a constant  $c(p) $ in the interval $[0, \gamma/p]$.
Then for all $T\geq 1 $ we have 
$$
\sum_{\substack{ a\leq T }} F(a)     G(a) \prod_{  p\mid a } (1+c(p ) )  
\leq 2^{\gamma \gamma'}
\sum_{\substack{ a\leq T  }} F(a) G(a)
,$$ where $  \gamma'=1+2(1+c_1)/c_2) c_0  C^{ 1+2(1+c_1) /c_2}+ C' (2^{c_2/2}-1)^{-1}
$. 
\end{lemma}

\begin{proof} 
Extending multiplicatively the function $c$ to positive     square-free integers  
we   get $$\prod_{\substack{   p\mid a  } } (1+c(p ) )  
= \sum_{\substack{ d\mid a    } } \mu(d)^2  c(d).$$   This  turns  the sum in the lemma into 
$$\sum_{\substack{ a\leq T }} F(a)     G(a)\sum_{\substack{ d\mid a    } } \mu(d)^2c(d)
=\sum_{\substack{ bd\leq T    }}  \mu(d)^2 c(d)  F(bd )     G(bd )  .$$
By assumption there exists $C'$ such that $G(ab ) \leq G(a) \min \{C^{\Omega(b)}, C' b^{c_2/2}\}$.
By Lemma~\ref{lem:postvenice} with $\epsilon=c_2/2$ we see that $G(n'n)\leq G(n') H(n)$ for all coprime $n,n'$, where 
$H$ is the multiplicative function given by $ H(p^e)= \min\{C^e, C'p^{e c_2/2}\}$ for  $e\geq 1$ and primes $p$. 
We factor $b=b_0b_1$, where $b_1$ is coprime to $d$ and each prime divisor of $b_0$ divides $d$. Thus,
$$  F(bd )     G(bd )=  F(b_0 b_1 d )     G(b_0 b_1 d )\leq  F(b_0  d ) F(b_1 )  
H(b_0 d)  G(  b_1   ),$$ hence, the sum is $$\leq \sum_{b_1\leq T}  F(b_1) G(b_1) 
\sum_{\substack{b_0 d \leq T/b_1 \\ p\mid b_0 \Rightarrow p\mid d \\ \gcd(b_1, d) = 1}}  
\mu(d)^2 c(d) F(b_0 d) H(b_0 d).$$ We will show that the inner double sum over $b_0$ and $d$ converges, and we will also upper bound the value that it attains. Dropping the condition $b_0d\leq T/b_1$ we can write it as $\prod_{p} (1+\mathcal E_p)$, where 
$$ \mathcal E_p= \sum_{\substack{  \beta,  \delta \in \Z\cap [0,\infty )   \\  (\beta , \delta) \neq (0,0)\\\beta >0 \Rightarrow \delta>0} }
\mu(p^\delta )^2  c(p^\delta )   F(p^{\beta+\delta } )  H(  p^{\beta +\delta } )
=  c(p) \sum_{  \beta \geq  0  }     F(p^{\beta+1 } )   H(  p^{\beta + 1} )
$$and  the product is taken over all primes $p\nmid d $. Let $\mathcal B$ be
the least integer satisfying $ 2(1+c_1) \leq (\mathcal B+1) c_2    $. To estimate the contribution of $\beta \leq \mathcal B$ 
we use    $c(p  ) \leq  \gamma/p$ to get  $$ \sum_{ 0\leq  \beta\leq \mathcal B   } c(p  )   F(p^{\beta+1 } )  H(  p^{\beta + 1} )
\leq \frac{\gamma}{p}  \sum_{ 0\leq  \beta\leq \mathcal B   }   F(p^{\beta+1 } )   C^{\beta +1} 
\leq  \frac{\gamma (\mathcal B+1) c_0  C^{ 1+\mathcal B} }{p^{2} } .$$ To bound the contribution of the remaining terms 
we use  $F(p^e) \leq p^{c_1-ec_2}$ to get $$ \sum_{    \beta\geq 1+ \mathcal B   } c(p  )   F(p^{\beta+1 } )  H(  p^{\beta + 1} ) 
\leq C' \gamma  p^{-1+c_1} \sum_{    \beta\geq 1+ \mathcal B   }      p^{-(\beta+1) c_2/2}.$$ This is at most $$
 C' \gamma  p^{-1+c_1} p^{-(\mathcal B+2) c_2/2}\sum_{    \beta\geq 0 }      2^{-(\beta+1) c_2/2}
=   \frac{C' \gamma (2^{c_2/2}-1)^{-1}}{p^{1-c_1+(\mathcal B+2) c_2/2 }}
 .$$ The exponent of $p$ in the right-hand side is strictly larger than $ 2 $ 
owing to our definition of $\mathcal B$. We have thus shown that 
for all primes $p$   one has $0\leq \mathcal E_p \leq \mathcal B' p^{-2}$, where 
$$\mathcal B': = \gamma (\mathcal B + 1) c_0  C^{ 1+\mathcal B} + C' \gamma (2^{c_2/2}-1)^{-1}.$$
By the definition of $\mathcal B$ we have $2(1+c_1) > \mathcal B  c_2  $, hence, $\mathcal B'\leq \mathcal D'$, where 
$$ \mathcal D' := \gamma (1+2(1+c_1)/c_2) c_0  C^{ 1+2(1+c_1) /c_2}+\gamma C' (2^{c_2/2}-1)^{-1},$$
hence, $\prod_{p}(1+\mathcal E_p) \leq
\prod_{p}(1+p^{-2 \mathcal D'} )\leq
 \prod_{p} (1+p^{-2} )^{\mathcal D'}\leq \zeta(2)^{\mathcal D'} \leq 2^{\mathcal{D}'}$.
\end{proof}

\begin{lemma} 
\label{lem-avergg} 
Fix   a positive constant $\alpha_1$   and let  $g:\mathbb N \to [0,\infty)$ be a multiplicative function for which 
$g(p) \leq \alpha_1/p$ for all primes $p$. Then for all $a\in \mathbb N , \alpha_2,\alpha_3>0$ and $x\geq 2$ we have 
$$ 
\sum_{\substack{ m\in \mathbb N, \gcd(m,a)=1 \\p\mid m \Rightarrow p\in (\alpha_1, x^{\alpha_2})}  }  
 \mu(m)^2 g(m )  \prod_{\substack{ \alpha_1< p\leq x^{\alpha_3} \\  p\nmid am}  } (1-g(p) )^2 \ll
 \mathcal C \prod_{\substack{ \alpha_1<  p\leq x^{\min\{\alpha_2,\alpha_3\}} \\  p\nmid a }  } (1-g(p) ) ,$$ where
$$ \mathcal C:=\mathds 1 [ \alpha_2 \leq \alpha_3]  \prod_{\substack{ x^{\alpha_2}<  p\leq x^{\alpha_3} \\  p\nmid a }  } (1-g(p) )^2 +
\mathds 1 [ \alpha_2 > \alpha_3]  \prod_{\substack{ x^{\alpha_3}<  p\leq x^{\alpha_2} \\  p\nmid a }  } (1-g(p) )^{-1} $$ and 
 the implied constant depends  on   $\alpha_1$ but is independent of 
$a, \alpha_2, \alpha_3$ and $x$. 
\end{lemma}

\begin{proof} 
Let $\mathcal P$ be the product of all primes in $(\alpha_1, x^{\alpha_2})$ that do not divide $a$. Using that $g$ is multiplicative and $g\geq 0$ we see that the sum over $m$ is 
\begin{align*} 
\sum_{m\mid \mathcal P  }   g(m )  & \prod_{\substack{ \alpha_1<  p\leq x^{\alpha_3} \\  p\nmid am}  } (1-g(p) )^2
= \prod_{\substack{ \alpha_1<  p\leq x^{\alpha_3} \\  p\nmid a }  } (1-g(p) )^2
\sum_{m\mid \mathcal P  }    g(m )\prod_{\substack{ \alpha_1<  p\leq x^{\alpha_3} \\    p\mid m }  } (1-g(p) )^{-2}
\\=  &\prod_{\substack{ \alpha_1<  p\leq x^{\alpha_3} \\  p\nmid a }  } (1-g(p) )^2
\prod_{\substack{ \alpha_1<  p\leq x^{\alpha_2} \\  p\nmid a, p> x^{\alpha_3} }  } (1+g(p) ) 
\prod_{\substack{ \alpha_1<  p\leq x^{\alpha_2} \\  p\nmid a, p\leq x^{\alpha_3} }  }    \left( 1+\frac{g(p)}{(1-g(p))^2} \right).
\end{align*}  
The assumption $g(p)\leq \alpha_1/p$  implies that $g(p)<1$ whenever $p> \alpha_1$, thus, 
we can use the approximations $$1+\epsilon=(1-\epsilon)^{-1} (1+O(\epsilon^2)), \quad 
 \left(1+ \frac{\epsilon}{(1-\epsilon)^2}\right)
=(1-\epsilon)^{-1} (1+O(\epsilon^2))$$ with $\epsilon=g(p)$ respectively in the second and third product. 
This will produce   $$  \ll  \prod_{\substack{ \alpha_1<  p\leq x^{\alpha_3} \\  p\nmid a }  } (1-g(p) )^2
  \prod_{\substack{  \alpha_1<  p\leq x^{\alpha_2} \\  p\nmid a }   } (1-g(p) )^{-1} 
\ll  \mathcal C \prod_{\substack{ \alpha_1<  p\leq x^{\min\{\alpha_2,\alpha_3\}} \\  p\nmid a }  } (1-g(p) ) $$ with an implied constants that depend at most on $\alpha_1$. This is because $\prod_p (1+O(g(p)^2 ) ) $ converges absolutely  due to the assumption $g(p)=O(1/p)$. 
\end{proof}

Let us recall a special case of~\cite[Lemma~6.3]{MR2061214} here:
\begin{lemma}
[Fundamental lemma of Sieve Theory]
\label{lem-fund-sieb} Let $\kappa>0, y>1$. There exists a sequence of real numbers $(\lambda_m^+)$ depending only on $\kappa$ and $y$
with the following properties: 
\begin{align}
\label{eq:siev1}\lambda_1^+&=1, \\ 
\label{eq:siev2}| \lambda_m^+|&\leq 1\ \textrm{ if } 1<m<y \\ 
\label{eq:siev3} \lambda_m^+&=0 \ \textrm{ if } m\geq y, 
\end{align}  and for any integer $n>1$,
\begin{equation}\label{eq:siev4}  0\leq \sum_{m\mid n } \lambda_m^+. \end{equation} Moreover, for any multiplicative function $f(m)$
with $0\leq f(p)<1$  and satisfying \begin{equation}\label{eq:sievdim}\prod_{w\leq p< z}(1-f(p))^{-1} 
\leq \left(\frac{\log z}{\log w}\right)^\kappa \left(1+\frac{K}{\log w}\right)\end{equation} for all $2\leq w < z\leq y $ we have 
\begin{equation}\label{eq:siev4bf7} \sum_{\substack{ m\mid P(z) }}\lambda_m^+f(m)=\left(
1+O\left(\mathrm e^{-\sigma} \left(1+\frac{K}{\log z}\right)^{10} \right)
\right)\prod_{p\leq z}(1-f(p) ),\end{equation} where $P(z)$ is the product of all primes $p\leq z$ and $\sigma=\log y/\log z \geq 1$, the implied constant depending only on $\kappa$.
\end{lemma}

\begin{lemma}
\label{lem-sivval}Let $g:\mathbb N\to [0,1)$ be as in Lemma~\ref{lem-avergg} and assume that 
there exist constants $\alpha_2, \alpha_3$ such that 
$$\prod_{w\leq p< z}(1-g(p))^{-1} 
\leq \left(\frac{\log z}{\log w}\right)^{\alpha_2} \left(1+\frac{\alpha_3}{\log w}\right)$$
 for all $2\leq w < z $. Fix any   constants $\xi_1,\xi_2 \in (0,1) $, $\Lambda_1, \Lambda_2>0$ and assume  that we are given a finite set of non-zero integers $\mathcal S=\{s_1,\ldots, s_N\}$
and a set of non-negative real numbers $x, a_1, \ldots, a_N$  
 such that  for all $d\leq x^{\xi_1}$  
one has \[ \sum_{\substack{1\leq n \leq N  \\  d\mid s_n } } a_n
 =g(d) x ( 1+\epsilon_1 )+\epsilon_2 ,\]
where $\epsilon_i$ are real numbers that satisfy 
 \[|\epsilon_1 |  \leq \Lambda_1 \prod_{\substack{ \alpha_1<p\leq x \\  p\nmid d  }} (1-g(p) )^2 
\ \ \mathrm{ and } \  \ |\epsilon_2  | \leq \Lambda_2 x^{1-\xi_2}.\]Fix any constants $\xi_3 \in (0,\xi_1)$ and $ \xi_4>0 $,
let $\Gamma= \max\{1/\xi_4, 1/(\xi_1-\xi_3),1/\xi_2\}$ and assume that $\log x> 4 \alpha_3 \Gamma$.

Then, for all $b\in \mathbb N$ satisfying  $b\leq x^{\xi_3} $   
we have  $$ 
\sum_{\substack{1\leq n \leq N, b\mid s_n   \\ p\leq x^{\xi_4} \  \mathrm{ and } \   p\nmid b \Rightarrow p\nmid s_n}} \hspace{-0,4cm} a_n 
\leq \mathcal C_0  \Big(\Gamma^{\alpha_2} x g(b) \prod_{\substack{ p\leq x\\p\nmid b }} (1-g(p ) ) +x^{1-\xi_2/2} \Big)
,$$ where   $\mathcal C_0 $ is a positive constant that is independent of $b, x$ and $\xi_4$.
\end{lemma}  

\begin{proof}  
Let $ \gamma=\min\{(\xi_1-\xi_3)/2,\xi_2/4 ,\xi_4 \}$.
We     employ  Lemma~\ref{lem-fund-sieb}   with 
$$ \kappa=\alpha_2, K=\alpha_3, y=x^{\min\{\xi_1-\xi_3,\xi_2/2\}}, f(p)=g(p) \mathds 1[p>\alpha_1 \ \& \ p\nmid b], z=x^\gamma,$$
where $\alpha_1$ is as in Lemma~\ref{lem-avergg}. To verify~\eqref{eq:sievdim}  we   note that  for  all $\alpha _1 < w' < z' $ one has 
$$ 
\prod_{w'\leq p< z'} \frac{1}{ 1-f(p) } =\prod_{\substack{ w'\leq p< z' \\ p\nmid b}}  \frac{1}{ 1-g(p) } \leq \prod_{ w'\leq p< z'   }  \frac{1}{ 1-g(p) } \leq 
\left(\frac{\log z'}{\log w'}\right)^{\!\alpha_2} \left(1+\frac{\alpha_3}{\log w'}\right).
$$  Define $\mathcal P$ to be the product of all primes $p\in (\alpha_1,z]$ that do not divide $b$. Then the cardinality in the lemma is bounded by
 $$
\sum_{\substack{1\leq n \leq N, b\mid s_n   \\  \gcd(s_n, \mathcal P )=1}}  a_n  
=
\sum_{\substack{1\leq n \leq N\\ b\mid s_n    }}   a_n \sum_{\substack{m\mid s_n \\  m\mid \mathcal P }} \mu(m)
 \leq 
\sum_{\substack{1\leq n \leq N\\ b\mid s_n    }}   a_n \sum_{\substack{m\mid s_n \\  m\mid \mathcal P }}  \lambda_m^+
= 
\sum_{m\mid \mathcal P } \lambda_m^+
\sum_{\substack{1\leq n \leq N\\ bm \mid s_n    }}   a_n 
 ,$$
where we used~\eqref{eq:siev1} and~\eqref{eq:siev4} in the     inequality.
By~\eqref{eq:siev2} the only $m$ that contribute must satisfy  $ bm \leq b y \leq  b x^{\xi_1-\xi_3}\leq x^{\xi_1}$. 
This allows us to use the assumption, thus,  
$$\sum_{m\mid \mathcal P } \lambda_m^+
\sum_{\substack{1\leq n \leq N\\ bm \mid s_n    }}   a_n 
=
xg(b)\sum_{m\mid \mathcal P} \lambda_m^+g(m)   + \epsilon_3 + \epsilon_4,$$
where we used~\eqref{eq:siev2} and the coprimality of  $b$ and $m$, and the $\epsilon_i$ are real numbers that satisfy 
$$ 
|\epsilon_3| \leq \Lambda_2 
 y x^{1-\xi_2}, \quad |\epsilon_4| \leq \Lambda_1  
x g(b) 
  \sum_{m\mid \mathcal P  }    g(m )  \prod_{\substack{ \alpha_1<p\leq x \\  p\nmid bm } } (1-g(p) )^2.
$$  
Our choice of $y$ makes sure that 
$y x^{1-\xi_2} \leq x^{1-\xi_2/2}$, which is acceptable. Note that $\xi_2<1$ hence $\gamma<1$. Thus, when applying 
  Lemma~\ref{lem-avergg} with $\alpha_2=\gamma, \alpha_3=1$ 
one sees that the factor $\mathcal C$ appearing in the lemma is at most $1$. 
This leads to the bound 
\[\leq \Lambda_3 \Big( xg(b) \big| \sum_{m\mid \mathcal P} \lambda_m^+g(m)  \big|  + x^{1-\xi_2/2} + xg(b)
  \prod_{\substack{\alpha_1< p \leq x^\gamma  \\p\nmid b  }}(1-g(p)  
\Big),\]  for some  positive real number $\Lambda_3$  that  is independent of $b,x$ and $\xi_4$.
Note that $g(m)=f(m)$ for all $m \mid \mathcal P$, thus, by~\eqref{eq:siev4bf7} we obtain 
$$
\big|\sum_{m\mid \mathcal P} \lambda_m^+g(m)\big|=\big|\sum_{m\mid \mathcal P} \lambda_m^+f(m) 
\big|\leq \Lambda_4 \prod_{\substack{ \alpha_1 < p \leq z  }} (1-f(p ) ) 
= \Lambda_4 \prod_{\substack{ \alpha_1 < p \leq x^\gamma \\ p\nmid b  }} (1-g(p ) ),
$$  
for some  positive real number $\Lambda_4$  that  is independent of $b,x$ and $\xi_4$.
We have so far obtained the bound $$\Lambda_5 \Big( x g(b) \prod_{\substack{ \alpha_1< p\leq x^\gamma \\p\nmid b }} (1-g(p ) ) +x^{1-\xi_2/2}
\Big)$$ 
for some  positive real number $\Lambda_5$  that  is independent of $b,x$ and $\xi_4$. It 
 remains to upper-bound the product over $p$. For this, we write 
$$
 \prod_{\substack{ \alpha_1< p\leq x^\gamma \\p\nmid b }} (1-g(p ) )  
\leq  \prod_{\substack{ \alpha_1< p\leq x  \\p\nmid b }} (1-g(p ) ) 
 \prod_{\substack{  x^\gamma< p \leq x    }} (1-g(p ) ) ^{-1} 
  $$ and use our assumptions to upper-bound it by 
 $$
\leq 
 \prod_{\substack{\alpha_1<  p\leq x  \\p\nmid b }} (1-g(p ) ) 
\left(\frac{\log x}{\log x^\gamma}\right)^{\alpha_2} \left(1+\frac{\alpha_3}{\gamma\log x }\right)
\leq \Lambda_6 \gamma^{-\alpha_2} \prod_{\substack{ \alpha_1< p\leq x  \\p\nmid b }} (1-g(p ) )
,$$ whenever $\log x> \alpha_3/\gamma$
and where $\Lambda_6$ is  a  positive real number  that  is independent of $b,x$ and $\xi_4$.
To conclude the proof note that 
$1/\gamma \leq 4 \Gamma$, hence, 
$ \gamma^{-\alpha_2}\leq (4\Gamma)^{\alpha_2} $
and $\log x> \alpha_3/\gamma$ due to $\log x> 4 \Gamma \alpha_3$.  
\end{proof}

\begin{lemma}\label{lem:lord of all fevers and plague}Fix any positive $c_0,c_1,c_2$, assume that 
  $F$ is as in Lemma~\ref{lemshiu} and that there exists $c_3\geq 0 $ such that for all primes $p$ and integers $e\geq 2 $ we have $F(p^e) \leq c_3/p^2$.
Fix any $C,C'>0$ and 
 assume that 
$G:\mathbb N\to [0,\infty)$ is a function such   
 that for all coprime positive integers $a,b$ one has  $G(ab ) \leq  G(a) \min\{C^{\Omega(b)}, C'b^{c_2/2}\}$.

Then  for all $x\geq 1 $ we have 
$$ \sum_{\substack{n\leq x\\P^-(n)> c_0}} F(n) G(n) \ll  \exp\left(\sum_{c_0< p\leq x} F(p) G(p) \right)
,$$ where the implied constant depends at most on $c_i$ and $C,C'$.
\end{lemma}

\begin{proof} We define a    multiplicative function  
$H'$ such that when $p$ is  prime and $e\geq 2 $ one has 
$H'(p^e)= \min\{C^e, C'p^{c_2 e/2}\} $ 
while $H'(p)= G(p)$. An easy modification of the proof of Lemma~\ref{lem:postvenice} shows that 
 for all coprime positive integers $a,b$  we have  $G(ab ) \leq  G(a) H'(b)$.
Hence, $G(b) \leq H'(b)$ for all $b$ and therefore the sum in the lemma is at most 
 $$ \sum_{\substack{n\leq x\\P^-(n)> c_0}} F(n) H'(n)  \leq 
\prod_{\substack{n\leq x\\P^-(n)> c_0}} \left(1+ \sum_{e\geq 1 } F(p^e) H'(p^e) \right)\leq \exp\l(\sum_{c_0<p\leq x, e\geq 1 } F(p^e) H'(p^e) \right)
 $$ due to the inequality $1+z\leq \mathrm e^z$ valid for all $z\in \mathbb R$.
Let $\mathfrak E$ be a positive integer that will be specified later. 
The contribution of  $e> \mathfrak E$ is at most 
$$p^{c_1} \sum_{e> \mathfrak E } p^{-e c_2}H'(p^e)
\leq 
C' p^{c_1} \sum_{e> \mathfrak E } p^{-e c_2/2} 
\leq C'  p^{c_1 - \mathfrak E c_2/2} (1-2^{- c_2/2} )^{-1}\ll p^{c_1 - \mathfrak E c_2/2}.
$$ Taking $  \mathfrak E$ to be the least positive integer satisfying $2(c_1 +2)/c_2 \leq  \mathfrak E  $ yields the bound $\ll p^{-2}$.
The contribution of the terms in the interval $[2, \mathfrak E]$ is 
$$\leq \sum_{2\leq e \leq \mathfrak E} F(p^e) H'(p^e)
\leq \sum_{2\leq e \leq \mathfrak E} F(p^e)  C^e
\leq \frac{c_3}{p^2} \sum_{2\leq e \leq \mathfrak E}   C^e \ll \frac{1}{p^2}\ll \frac{1}{p^2}.$$ 
Thus, the overall bound becomes 
$$
 \exp\l(\sum_{c_0<p\leq x, e\geq 1 } F(p^e) H'(p^e) \right)
\leq  \exp\l( \sum_{c_0<p\leq x  } F(p) H'(p) \right)\exp\l( \sum_{c_0<p\leq x  } O(1/p^2) \right)
,$$ which is sufficient because $H'(p)=G(p)$.
\end{proof}

\section{The upper bound}\label{s:thank msri for not letting me go to the workshop that started the project}

\subsection{Start of the proof} 
Let us define the constants 
\begin{equation}
\label{chapelofgouls}
\eta_1 := \frac{1}{\alpha} \min\left\{\frac{\xi}{20}, \frac{\theta}{2}, \frac{1}{2} \right\},
\ \ \eta_2: = \frac{1}{2},
\ \ \eta_3 :=\min\left\{\frac{ \lambda_1    \eta_2   }{2( 1+\lambda_1+\lambda_2)  } ,  \frac{1}{2} \right\}
.\end{equation}  
Define\begin{equation}
\label{def:immortal rites mMA}Z := M^{\alpha \eta_1}.\end{equation}  For $a\in \mathcal A$  
we factorise $c_a = p_1^{e_1} \cdots p_r^{e_r}$ with primes $p_1 < \dots < p_r$ and exponents $e_i\geq 1$. Let  $d_a$   
be the unique integer of the form $d_a  := p_1^{e_1} \cdots p_i^{e_i}$ satisfying
\begin{equation}\label{eq:oboe}p_1^{e_1} \cdots p_i^{e_i} \leq Z < p_1^{e_1} \cdots p_i^{e_i} p_{i + 1}^{e_{i + 1}}
\end{equation} and let $b_a := p_{i + 1}^{e_{i + 1}} \cdots p_r^{e_r}$. 
By construction we have\begin{align} &P^+(d_a) < P^-(b_a), \\
\label{eGCD}&\gcd(d_a, b_a ) = 1, \\
& d_a \leq Z.\end{align} The following cases will be considered:
\begin{enumerate}[label=(\roman*)]
\item $P^-(b_a) \geq Z^{\eta_3}$,
\item $P^-(b_a) < Z^{\eta_3}$ and $d_a \leq Z^{1 - \eta_2}$, 
\item $P^-(b_a) \leq ( \log Z)  \log \log Z$ and $Z^{1 - \eta_2} < d_a \leq Z$,
\item $(\log Z) \log \log Z < P^-(b_a) < Z^{\eta_3}$ and $Z^{1 - \eta_2} < d_a \leq Z$.
\end{enumerate}  
\subsection{Case (i)}
\label{case1}
The plan in this case is to show that $b_a$ has few prime divisors so that $c_a $  has few prime divisors in a large interval.
The density of $a $ with the latter property will be bounded by the Brun sieve.
 
For the $a\in \mathcal A$ in the present case  we have
\[M^{\alpha \eta_1 \eta_3 \Omega(b_a)} =  Z^{\eta_3\Omega(b_a)}
\leq P^-(b_a)^{\Omega(b_a)} \leq b_a \leq  c_a  \leq \widetilde{B} M^\alpha\]
and therefore $\Omega(b_a) \leq \frac{1+\log \widetilde{B}}{\eta_1 \eta_3}$ for $M > e^{1/\alpha}$. 
By~\eqref{eGCD} we have $\gcd(d_a, b_a) = 1$, thus leading via Definition~\ref{dMultRegion} to
\[ f(c_a )  \leq f(d_a ) A^{\frac{1+\log  \widetilde{B} }{\eta_1 \eta_3} }.\]
Now let $d:=d_a$, so that $d\leq Z $ and $d\mid c_a$.
Furthermore,  $c_a$ is coprime to every prime   in the interval $[2,Z^{\eta_3})$ that does not divide $d$. 
This is because every   prime  that divides $c_a$  must necessarily divide $d_a $ or $b_a$
and in our case all prime divisors of $b_a$ are in the interval $[Z^{\eta_3}, \infty)$.
In particular, $c_a $ is coprime to every prime   in  the interval $(B,Z^{\eta_3})$ that is coprime to $d$.
Define  $$\mathcal{P} :=\prod_{\substack{ p \in (B,Z^{\eta_3}) \\ p\nmid d  } } p.$$  
We obtain   $$\sum_{\substack{ a\in \mathcal A \\ \mathrm{ case \  (i)   } } }  \chi_T(a) f(c_a) 
\leq  A^{\frac{1+\log \widetilde{B} }{\eta_1 \eta_3} }  \sum_{ d \leq Z }    f(d)  
\sum_{\substack{ a\in \mathcal A, d \mid c_a  \\ \gcd(\mathcal{P} , c_a ) = 1 } }  \chi_T(a)  . $$
To deal with   the coprimality condition  
we     employ Lemma~\ref{lem-sivval}  with 
$$ \mathcal S=\{c_a: a\in \mathcal A, \chi_T(a)>0\}, 
\{a_n:1\leq n \leq N\}=\{\chi_T(a):a\in \mathcal A, \chi_T(a)>0\}
 $$ and 
$ x=M, g=h, \alpha_1=B, \alpha_2=\kappa, \alpha_3=K, \xi_1=\theta, \xi_2=\xi, b=d, \xi_3=\alpha \eta_1, {\xi_4}={\alpha \eta_1 \eta_3}.$ The assumption $\xi_3<\xi_1$ 
is satisfied due to~\eqref{chapelofgouls}. Thus, $$\sum_{\substack{ a\in \mathcal A, d \mid c_a  \\ \gcd(\mathcal{P} , c_a ) = 1 } }  \chi_T(a) 
\ll  M 
h_T(d)
  \prod_{\substack{ B<p\leq M\\p\nmid d }} (1-h_T(p) ) +M^{1-\xi/2},$$ 
 where  the implied constant  is independent of $d,M$ and $ T$
but is allowed to depend on  $ \alpha,  \eta_1,  \eta_3, K, \kappa, \lambda_i, \theta$ and $\xi$. This gives the overall bound  
$$\sum_{\substack{ a\in \mathcal A \\ \mathrm{ case \  (i)   } } }  \chi_T(a) f(c_a) 
\ll  A^{\frac{1+\log \widetilde{B} }{\eta_1 \eta_3} } \sum_{ d \leq Z }    f(d)
\l\{M h_T(d ) \prod_{\substack{ B<p\leq M\\p\nmid d }} (1-h_T( p) ) +M^{1-\xi/2}
\r\}.$$ 
Since $f(n)\ll n$,
we infer that 
$$
\sum_{d \leq Z}  f(d) M^{1-\xi/2}\ll Z^{2}M^{1-\xi/2}\ll M^{2\alpha \eta_1+1-\xi/2} \leq M^{ 1-\xi/3}
$$  
due to~\eqref{chapelofgouls}. This leads us to 
$$
\sum_{\substack{ a\in \mathcal A \\ \mathrm{ case \  (i)   } } }  \chi_T(a) f(c_a)
 \ll M   \sum_{ d \leq Z }    f(d) h_T( d) \prod_{\substack{ B<p\leq M\\p\nmid d }} (1-h_T(p ) )  
+M^{ 1-\xi/3}.
$$  
We can now extend the sum over $d$  to all $d\leq M$ due to~\eqref{chapelofgouls} that guarantees that $Z\leq M$.
Combining this together with Lemma~\ref{lem:damnnotfinal} for $F=h_T$, $G=f$ and $c(p) = -1 + (1 - h_T(p))^{-1}$ yields 
\begin{equation}
\label{visionsfromthedarkside}
\sum_{\substack{ a\in \mathcal A \\ \mathrm{ case \  (i)   } } }  \chi_T(a) f(c_a) \ll 
M \prod_{\substack{ B<p\leq M  }} (1- h_T( p))  
  \sum_{ d \leq M }    f(d) h_T(d)  +
M^{ 1-\xi/3}.
\end{equation} 
\subsection{Case (ii)} \label{case2}
The main idea is to show that the exponent of $P^-(b_a)$ in the prime factorisation of $c_a$ is large
and then prove that this cannot happen too often.

Let $q:=P^-(b_a)$. Equation~\eqref{eq:oboe} and  the definition  of case (ii) respectively show 
$$  
Z < d_a  q^{v_q(d_a)}, d_a \leq Z^{1-\eta_2},
$$ 
thus, $q^{v_q( c_a)}>Z^{\eta_2}$. For a prime $p$, we take $m_p$ to be the smallest positive integer such that $p^{m_p} > Z^{\eta_2}$ and we take $n_p$ to be the largest positive integer such that $p^{n_p} \leq M^{\theta}$. We set $f_p = \min(m_p, n_p)$. Then we always have
\begin{equation}
\label{eq:tocata913}
p^{ f_p}>  \frac{1}{p}  M^{\min\{\alpha \eta_1 \eta_2 , \theta \}}
= \frac{M^{ \alpha \eta_1 \eta_2  }}{p}  .
\end{equation}
Also observe that $q^{f_q} \mid c_a$ (by $q^{f_q} \mid q^{m_q}$ and $q^{m_q} \mid c_a$) and $q^{f_q}\leq M^\theta$. 
Thus, we have shown that there exists a prime 
$q<Z^{\eta_3}$ (due to  the definition  of case (ii)) 
that has the properties $q^{f_q}\mid c_a$, $q^{f_q}\leq M^\theta$ and~\eqref{eq:tocata913}. Hence, by Definition~\ref{def:levdistr} we obtain  
$$\sum_{\substack{ a\in \mathcal A \\    \mathrm{ case  \  (ii) }  } }    \chi_T(a)  \leq \sum_{\textrm{prime } q<Z^{\eta_3}}  C_{q^{f_q}}(T)  
\ll  \sum_{\textrm{prime } q<Z^{\eta_3}} (h_T(q^{f_q})   M+ M^{1-\xi })\leq M \mathcal  S +Z^{\eta_3}M^{1-\xi}  , $$ 
where  $ \mathcal S := \sum_{q<Z^{\eta_3}}  h_T(q^{f_q}) 
$.  By~\eqref{eHighPower}  and~\eqref{eq:tocata913}  the sum 
 $ \mathcal S $ is at most 
$$
\sum_{q<Z^{\eta_3}}    q^{-f_q \lambda_1 + \lambda_2} \leq 
M^{-\lambda_1 \alpha \eta_1 \eta_2}  \sum_{q<Z^{\eta_3}}     q^{  \lambda_1+ \lambda_2} 
\leq M^{-\lambda_1 \alpha \eta_1 \eta_2} Z^{\eta_3(1+\lambda_1+\lambda_2)}.
$$ 
This equals $M^{-\rho}$, where 
$$  
\rho :=  \lambda_1 \alpha \eta_1 \eta_2 - {\alpha \eta_1 \eta_3(1+\lambda_1+\lambda_2)}
=\alpha \eta_1( \lambda_1   \eta_2 -  \eta_3(1+\lambda_1+\lambda_2))
$$  is strictly positive owing 
to~\eqref{chapelofgouls}. Fix any $\delta>0$. By Definition~\ref{dMultRegion}   we have $f(c_a) \leq C c_a^{\delta/\alpha}$ for 
all $a\in \mathcal A$, where  $C$ is positive and   depends on  $\alpha$ and $\delta$.
Thus,~\eqref{Roll Over Beethoven} shows that for all $ a\in \mathcal A $ one has  $f(c_a) \ll CM^{  \delta }$. 
We have therefore  proved that for every $\delta>0$ one has 
\begin{equation}\label{abhishek}
\sum_{\substack{ a\in \mathcal A  \\  \mathrm{ case  \  (ii)  }
 } }  \chi_T(a)  f(c_a) \ll
  C\left(M^{1-\rho+\delta} +Z^{\eta_3} M^{1-\xi+\delta}\right)
=      C\left(M^{1-\rho+\delta} + M^{1-\xi+\delta+\alpha \eta_1 \eta_3 }\right)\ll CM^{1+\delta-\beta_1}
,\end{equation} where 
 $\beta_1:= \min\left\{ \alpha \eta_1( \lambda_1 \eta_2 -  \eta_3(1+\lambda_1+\lambda_2)),  \xi -\alpha \eta_1 \eta_3\right\}
 $ is positive due to~\eqref{chapelofgouls} and the fact that $\eta_3<1$. Furthermore,
the implied constant depends at most  on  $ \alpha,  \widetilde{B},\delta,  K, \kappa, \lambda_i,  \theta$ and $\xi$.

\subsection{Case (iii)}  \label{case3}
The key idea in this case is to show that  $d_a$ is divisible    only by very small primes and then show that this does not 
happen too often. We have $$
\sum_{\substack{ a\in \mathcal A \\  \mathrm{ case  \  (iii)  }
 } } \chi_T(a) \leq \sum_{\substack{Z^{1 - \eta_2} < d \leq Z   \\ P^+(d) \leq (\log Z) \log \log Z}} 
\sum_{\substack{  a\in \mathcal A  \\    d\mid c_a  } } \chi_T(a) = \sum_{\substack{Z^{1 - \eta_2} < d \leq Z   \\ P^+(d) \leq (\log Z) \log \log Z}} 
C_d(T).$$ Equation~\eqref{chapelofgouls}  makes sure that $d\leq Z\leq M^\theta$, thus,  
we can employ the estimate in Definition~\ref{def:levdistr}.
It yields the upper bound
$$
\ll 
\sum_{\substack{Z^{1 - \eta_2} <  d \leq Z   \\ P^+(d) \leq (\log Z)  \log \log Z}}   (   M^{1-\xi}+h_T(d) 
 M)
\leq Z    M^{1-\xi}+
M \sum_{\substack{Z^{1 - \eta_2} < d \leq Z   \\ P^+(d) \leq (\log Z)  \log \log Z }}   h_T(d) 
.$$ 
To bound the sum over $d$ we   employ  Lemma~\ref{lemshiu} with $$F= h_T, c_0=B, c_1=\lambda_2, c_2=\lambda_1, x=Z, z=Z^{1-\eta_2}.$$
It shows that the sum over $d$ is  
$$\ll Z^{-(1-\eta_2) c} M^{o(1)}=M^{-\alpha \eta_1 (1-\eta_2) c+o(1)} \leq M^{-\alpha \eta_1 (1-\eta_2) c/2}
,$$ where $$c:=  \min\left\{ \frac{\lambda_1}{2} , \frac{1}{1+[2\lambda_2/\lambda_1]}\right\}
.$$
The overall bound becomes $$ \ll 
Z   M^{1-\xi}+  M^{1-\alpha \eta_1 (1-\eta_2) c/2}
=   M^{1-\xi+\alpha \eta_1}+  M^{1-\alpha \eta_1 (1-\eta_2) c/2}
\ll M^{1-\beta_2},$$
where $$ \beta_2 :=\min\{\xi-\alpha \eta_1,\alpha \eta_1 (1-\eta_2) c/2 \} $$
is   strictly positive by~\eqref{chapelofgouls} and the fact that $\eta_2 \in (0,1)$.
Bringing everything together we conclude that 
  for every $\delta>0$ one has 
\begin{equation}
\label{raphael}
\sum_{\substack{ a\in \mathcal A \\  \mathrm{ case \  (iii)   }
 } }  \chi_T(a) f(c_a) \ll M^{1+\delta-\beta_2}.
\end{equation}  

\subsection{Case (iv)} \label{case4}
The main idea is  to use the fact that   $c_a/d_a$ has no small prime divisors and then apply 
  the Brun sieve to see that this can happen with low probability, even when one counts with the additional
weight $A^{\Omega(c_a/d_a)}$.
 
Recalling~\eqref{eGCD} and Definition~\ref{dMultRegion} we     see    that 
$$f(c_a) = f(d_a b_a) \leq   f(d_a) A^{\Omega(b_a)} .$$
Thus, letting $d=d_a $, we infer that  
  \begin{equation}\label{eq:umbertoD}\sum_{\substack{ a\in \mathcal A  \\  \mathrm{ case \   (iv)} 
 } }   \chi_T(a) f(c_a)
\ll \sum_{Z^{1 - \eta_2} < d \leq Z} f(d) 
\Osum_{\substack{ a\in \mathcal A \\  d \mid c_a
 } }  \chi_T(a)    A^{\Omega(c_a/d)},
 \end{equation}
where $\Osum$ is subject to the further conditions 
$$   \gcd(d,c_a/d)=1  \ \textrm{ and }   \ (\log Z) \log \log Z < P^-(c_a/d)< Z^{\eta_3}.$$
It would be easier to estimate the sum over $a$ in the right-hand side of~\eqref{eq:umbertoD} if the summand $A^{\Omega(c_a/d)}$
was a constant. With this in mind we freeze the value of $P^-(c_a/d)$ as follows:
let $$ s:= \left[\frac{\log Z}{\log P^-(c_a/d)}\right]$$
so that   $ Z^{1/(s+1)} < P^-(c_a/d) \leq  Z^{1/s}  $
and $s \in \mathbb N \cap [1,s_0]$, where 
$$
s_0:= \left[\frac{\log Z}{\log \{(\log Z)(\log \log Z)\}}\right] \leq \frac{\log Z}{\log \log Z}
$$
for $Z$ large enough. By~\eqref{Roll Over Beethoven}  we have for $a$ with $\chi_T(a)\neq 0$ that  
$$ M^{\alpha \eta_1 \frac{\Omega(c_a/d)}{s+1}}  =\big(Z^{1/(s+1)} \big)^{\Omega(c_a/d)}
< P^-(c_a /d)^{\Omega( c_a /d)} \leq c_a /d  \leq c_a\leq \widetilde B M^\alpha  $$ thus, 
for $M\geq \mathrm e $ we obtain $$     \Omega(c_a /d) 
\leq (s+1) \left(  \frac{1}{\eta_1 } +\frac{ \log \widetilde B}{\alpha\eta_1 } \right) \leq 2s \left(  \frac{1}{\eta_1 } +\frac{ \log \widetilde B}{\alpha\eta_1 } \right)= 
\tau   s, $$ where $\tau =\tau (\alpha, \widetilde B, \eta_1)$ is a positive constant.
Hence the right-hand side of~\eqref{eq:umbertoD} is $$\ll
\sum_{1\leq s \leq s_0 } A^{\tau s }
\sum_{\substack{ Z^{1 - \eta_2} < d \leq Z \\ P^+(d)<Z^{1/s}  }} f(d)  
\sum_{\substack{a\in \mathcal A, d\mid c_a, \gcd(d,c_a/d)=1 \\Z^{1/(s+1)} < P^-(c_a/d) \leq  Z^{1/s}  }} \chi_T(a) 
.$$
The sum over $a$ is at most 
$$\sum_{\substack{a\in \mathcal A, d\mid c_a \\
p\leq Z^{1/(s+1)} \  \mathrm{and}   \  p\nmid d \Rightarrow p\nmid  c_a
 }} \chi_T(a) ,$$which will be bounded by    employing Lemma~\ref{lem-sivval} with 
\begin{align*}
\mathcal S&=\{c_a: a\in \mathcal A, \chi_T(a)>0\}, 
\{a_n:1\leq n \leq N\}=\{\chi_T(a):a\in \mathcal A, \chi_T(a)>0\},
\\
g&=h_T, \alpha_1=B,  \alpha_2=\kappa,
\alpha_3=K, x=M, \xi_1=\theta, \xi_2= \xi, \xi_3=\frac{\theta}{2}, \xi_4=\frac{\alpha \eta_1  }{s+1}, b=d, 
\end{align*}
where $h_T, B, \kappa, K, \theta, M$ and $ \xi$ are as in Definition~\ref{def:levdistr}. The assumption $b\leq x^{\xi_3}$ of Lemma~\ref{lem-sivval} is satisfied  due 
to~\eqref{chapelofgouls}. 
 The further assumption $\log x> 4 \alpha_3 \Gamma$ is satisfied for all large enough $M$ compared to 
$K, \alpha, \eta_1, \theta, \xi$ due to the inequality $$ \Gamma = \max\left\{\frac{1+s}{\alpha \eta_1},\frac{ 2}{\theta},\frac{1}{\xi}\right\}
\ll_{\alpha, \eta_1, \theta, \xi} 1+s \leq  1+s_0 \leq 1+ \frac{\log Z}{\log \log Z} \ll_{\alpha ,\eta_1} \frac{\log M}{\log \log M}.$$
We obtain the upper bound 
\begin{align*}
&\ll  \max\left\{\frac{1+s}{\alpha \eta_1},\frac{ 2}{\theta},\frac{1}{\xi}\right\}^\kappa \!\!\! M h_T(d)
\prod_{\substack{B< p\leq M \\p\nmid d  } } (1-
h_T(p)
 ) +M^{1-\xi/2}
\\
&
\ll   
s ^\kappa M 
h_T(d)
 \prod_{\substack{B< p\leq M \\p\nmid d  } } (1-
h_T(p)
) +M^{1-\xi/2},\end{align*}
where the implied constants 
are independent of $s,d$ and $M$. 
Thus, the right-hand side of~\eqref{eq:umbertoD} is $$\ll M
\sum_{1\leq s \leq s_0 } A^{\tau s }    s ^\kappa
\sum_{\substack{ Z^{1 - \eta_2} < d \leq Z \\ P^+(d)<Z^{1/s}  }} f(d)  
h_T(d)
  \prod_{\substack{ B<p\leq M \\p\nmid d  } } (1- h_T(p)
)  +M^{1-\xi/2}
\sum_{1\leq s \leq s_0 } A^{\tau s }
\sum_{  d \leq Z    } f(d)  
.$$ We have $\sum_{  d \leq Z    } f(d)  \ll Z^{2}=M^{2\alpha \eta_1 }$ 
by Definition~\ref{dMultRegion}. Thus, 
$$
M^{1-\xi/2}
\sum_{1\leq s \leq s_0 } A^{\tau s }
\sum_{  d \leq Z    } f(d)  
\ll M^{1-\xi/2+2\alpha \eta_1 } s_0  A^{\tau s_0 }  \leq M^{1-\xi/3} $$ 
due to~\eqref{chapelofgouls} and
the inequality   $s_0\leq (\log Z)/(\log \log Z)$  which implies that 
$$
 s_0  A^{\tau s_0 } \ll     A^{2\tau s_0 } =Z^{O(1/\log \log Z)}=M^{o(1)}
.$$ 
Thus,   the right-hand side of~\eqref{eq:umbertoD} is $$\ll M
\sum_{1\leq s \leq s_0 } A^{\tau s }    s ^\kappa
\sum_{\substack{ Z^{1 - \eta_2} < d \leq Z \\ P^+(d)<Z^{1/s}  }} f(d)  h_T(d)
 \prod_{\substack{ B<p\leq M \\p\nmid d  } } (1-  h_T(p)
 )  +M^{1-\xi/3 }
.$$ By~\eqref{chapelofgouls} we have  $\alpha \eta_1 \leq 1$,  so that $d \leq Z \leq M$.
  Then the product over $p $ is       
$$
\leq \prod_{\substack{ B<p\leq M \\p\nmid d  } } (1-  h_T(p)
) 
=\prod_{\substack{ B<p\leq M   } } (1- h_T(p) ) 
\prod_{\substack{ B<p\\ p\mid d  } } (1- h_T(p)) ^{-1} 
$$ and we get the bound  $$\ll M\prod_{\substack{ B<p\leq M   } } (1-h_T(p) ) 
\sum_{1\leq s \leq s_0 } A^{\tau s }    s ^\kappa
\sum_{\substack{ Z^{1 - \eta_2} < d \leq Z \\ P^+(d)<Z^{1/s}  }} f(d) h_T(d) \prod_{\substack{ B<p \\ p\mid d  } } (1-h_T(p) ) ^{-1}  +M^{1-\xi/3 }
.$$  We now bound the sum over $d$ by alluding to Lemma~\ref{lem:final} with 
$$\Upsilon=Z^{1 - \eta_2}, \Psi=Z^{1/s}, F=h_T, G=f, c_0=B, c_1=\lambda_2, c_2=\lambda_1, \varpi=\beta_0, C=A, 
$$ where $\varpi$ is defined via 
$
4 A^\tau=   \mathrm e^{  \varpi   (1 - \eta_2)  }
$. This means that $\varpi$ depends on $ \alpha, A,\widetilde B, \eta_1, $ and $ \eta_2$.
Hence, the sum over $d$ is 
$$  
 \ll    \exp (   -   \varpi s (1 - \eta_2)  )
\sum_{\substack{ d\leq Z^{1/s}  }} f(d) h_T(d) \prod_{\substack{ B<p\\ p\mid d  } } (1-h_T(p) ) ^{-1}.
$$  
We can extend the summation to all $d\leq M$ since the summand is non-negative and $Z^{1/s} \leq Z \leq M$. 
Thus,  the right-hand side of~\eqref{eq:umbertoD} is 
$$
\ll M\prod_{\substack{ B<p\leq M    } } (1-h_T(p)) 
\sum_{\substack{d \leq M}} f(d)    h_T(d)\prod_{\substack{ B<p\\ p\mid d  } } (1-h_T(p) ) ^{-1}
\sum_{1\leq s \leq s_0 } z^{s} s ^\kappa +M^{1-\xi/3 }
,$$   
where $ z= A^{\tau  }  \mathrm e^{     -\varpi   (1 - \eta_2)  }$. 
By the definition of $\varpi$ we have $z=1/4$, hence, the sum over $s$ is bounded in terms of $\kappa$.
Thus, we have shown that 
$$ 
\sum_{\substack{ a\in \mathcal A \\  \mathrm{ case \  (iv)   }
 } } \chi_T(a)   f(c_a) \ll  M\prod_{\substack{ B<p\leq M    } } (1-h_T(p)) 
\sum_{\substack{d \leq M}} f(d) h_T(d) \prod_{\substack{ B<p\\ p\mid d  } } (1-h_T(p)) ^{-1}
  +M^{1-\xi/3 },
$$  
where  the implied constant depends at most  on   $\alpha,  A,  \widetilde B, B, \lambda_i , \eta_i, \theta, \xi$ and $ \kappa$. Alluding to Lemma~\ref{lem:damnnotfinal} with  $F=h_T$ and $ G=f$ yields 
 \begin{equation}
 \label{eq:fredrik}
\sum_{\substack{ a\in \mathcal A \\  \mathrm{ case \  (iv)   }
 } }  \chi_T(a)  f(c_a) \ll  M\prod_{\substack{ B<p\leq M    } } (1-h_T(p)) 
\sum_{\substack{ d\leq M  }} f(d)    h_T(d) 
  +M^{1-\xi/3 }.
  \end{equation}  
 
\subsection{Proof of Theorem~\ref{tMain}}
The upper bound claimed in Theorem~\ref{tMain}
derives from~\eqref{visionsfromthedarkside} and~\eqref{eq:fredrik}. Taking $\delta=\beta_1/2$ in~\eqref{abhishek} and $\delta=\beta_2/2$ in~\eqref{raphael} shows that 
cases (ii) and (iii) contribute 
$\ll M^{1-\beta_3}$, where $\beta_3$ is given by 
$$  \min\left\{
\frac{\alpha \eta_1( \lambda_1 \eta_2 -  \eta_3(1+\lambda_1+\lambda_2)) }{2} 
,  
 \frac{\xi-\alpha \eta_1}{2}
,
\frac{\alpha \eta_1 (1-\eta_2) \lambda_1}{8}
,
\frac{\alpha \eta_1 (1-\eta_2) }{4(1+[2\lambda_2/\lambda_1])}
\right\} 
.$$
The term $M^{ 1-\xi/3} $ that is present in~\eqref{visionsfromthedarkside} and~\eqref{eq:fredrik} and the term $M^{1 - \beta_3}$ may be absorbed in the upper bound from Theorem~\ref{tMain}, this concluding the proof.
 
\section{The lower bound}\label{s:lowerbnds}
Recall the notation of $\theta, \xi $ in    Definition~\ref{def:levdistr}
and let $\kappa, K$ be as in Definition~\ref{densfnct}.
We introduce the constants $$v=\min\left\{1,
\frac{ \theta \min\{1/4,\xi/(4\theta)\} }{1+9\kappa +(\log 2)+10 (\log K )} 
\right\}, v_0 :=\min\{v/2,\theta/2\}.$$
Let $z:=M^v$.
For each $c\in \mathbb N$ we define $$ c^{\flat}=
\prod_{p\leq z}  p^{v_p(c)}.$$ 
Note that for a  positive integer $d$ satisfying $P^+(d) \leq z$, 
one has $d=c^\flat $  if and only if $d$ divides $c$ and the smallest prime divisor of $c/d$ strictly exceeds $z$.
Classifying all $a\in \mathcal A$ according to the value of $d:=c_a^\flat$ we thus obtain 
$$\sum_{a\in \mathcal A} \chi_T(a) f(c_a)= \sum_{\substack{d\in \mathbb N \\ P^+(d)\leq z }} \sum_{\substack{a\in \mathcal A \\ c_a^\flat=d}} \chi_T(a)f(c_a)
= \sum_{\substack{d\in \mathbb N \\ P^+(d)\leq z }} \sum_{\substack{a\in \mathcal A, d\mid c_a \\ 
 P^-(c_a/d)>z}} \chi_T(a) f(c_a) .$$
Note that if $d\leq M^{v_0}$ then $P^+(d)\leq d \leq  M^{v_0} \leq z $. 
Thus, since   $f\geq 0 $, we can restrict the sum over $d$ to get the lower bound 
$$\sum_{a\in \mathcal A }\chi_T(a) f(c_a) \geq  \sum_{1\leq d \leq M^{v_0}  }  \sum_{\substack{a\in \mathcal A, d\mid c_a \\ P^-(c_a/d)>z}} \chi_T(a) f(c_a) .$$
Using~\eqref{Roll Over Beethoven} and  the inequality $m\geq P^-(m)^{\Omega(m)}$  
 leads to $$\Omega(c_a/d) \leq \frac{\log (c_a/d)}{\log P^-(c_a/d)}  \leq \frac{\log c_a}{\log P^-(c_a/d)} 
\leq \frac{\log (\widetilde B M^\alpha )}{\log  (M^v)}\leq L_0$$ for 
some $L_0=L_0(\widetilde{B} ,\alpha, v) >0$. Therefore, by assumption, $f(c_a/d)\gg 1$, where the implied constant depends at most on $L_0$.   
Since $P^+(d) < P^-(c_a/d)$ we see that $d, c_a/d$ are coprime, hence, the multiplicativity of $f$ yields 
$$ 
f(c_a)=f(d) f(c_a/d)\gg_{\widetilde B ,\alpha, v} f(d). 
$$ 
Injecting this into the previous estimate will yield 
\begin{equation}
\label{eq:canipleasesithere?} 
\sum_{a\in \mathcal A } \chi_T(a) f(c_a)\gg_{\widetilde B ,\alpha, v}  \sum_{1\leq d \leq M^{v_0}  } f(d)
\sum_{\substack{a\in \mathcal A ,  d\mid c_a \\  P^-(c_a/d)>z }} \chi_T(a) .
\end{equation} 
We will now lower bound the sum over $a\in \mathcal A $ by arguments similar to the ones in the 
proof of Lemma~\ref{lem-sivval}. Using the sequence $\lambda_m^-$ from~\cite[Lemma~6.3]{MR2061214}
we obtain  $$\sum_{\substack{a\in \mathcal A ,  d\mid c_a \\  P^-(c_a/d)>z }}  \chi_T(a) 
\geq  \sum_{m\mid \mathcal P} \lambda_m^-
\sum_{\substack{a\in \mathcal A \\   dm \mid c_a   }}  \chi_T(a) 
,$$ where $\mathcal P$ is the product of all primes $p\leq z$. 
Recall from~\cite[Lemma~6.3]{MR2061214} that $\lambda_m^-$ is supported on integers $m\leq y $.
We define   $y=M^{\theta \epsilon}$ where  $\epsilon=\min\{1/4,\xi/(4\theta)\}$.
Then  the only $m$ that contribute to the sum satisfy  
$$
d m \leq  d y= d M^{\theta \epsilon }
\leq M^{v_0+\theta \epsilon }\leq  M^{\theta/2+\theta \epsilon }\leq M^\theta,
$$ 
thus, we can use the assumption on the growth of $C_d(T)$. It yields the estimate  $$
M
h_T(d) 
\sum_{m\mid \mathcal P} \lambda_m^-
h_T(m) 
+o\left(
Mh_T(d) \sum_{\substack{ m\mid \mathcal P \\ m\leq y  } } h_T(m) 
 \prod_{\substack{ B<p \leq M\\ p\nmid dm } } (1-h_T(p) )^{2} 
\right)+O\left(
M^{1-\xi } \sum_{   m\leq y  } 1  
\right)
.$$ The last error term is $  \ll M^{1-\xi }  y= M^{1-\xi+\epsilon \theta}$. Since  
 $\epsilon \theta < \xi/2$, the error term becomes  $O(M^{1-\xi/2}) $, which is acceptable.

By taking out the largest factor of each $m\mid \mathcal{P}$ that is a product of primes that satisfy $p\leq B$ or $p\mid d$, the sum over $m$ in the error term is
\[\leq \prod_{p\leq B}(1+
h_T(p) 
)\prod_{B<p\mid d}(1+
h_T(p) 
)
\sum_{\substack{  m\leq y,\gcd (m,d)=1 \\p\mid m\Rightarrow B<p\leq z } } \mu(m)^2
h_T(m) 
 \prod_{\substack{ B<p \leq M\\ p\nmid dm } } (1-
h_T(p) 
)^{2}.\]
 The primes $p\leq B$ contribute $O_B(1)$. 
Using Lemma~\ref{lem-avergg} with $\alpha_1=B$, $x^{\alpha_2}=z=M^v$, $x^{\alpha_3}=M$ and $a=d$,
and taking advantage of the fact that  $v\leq 1 $, we infer that
$$\ll\prod_{B<p\mid d}(1+
h_T(p) 
)
\prod_{\substack{M^v<p<M\\ p\nmid d}}(1-
h_T(p) 
)^2
\prod_{\substack{B<p<M^v\\ p\nmid d}}(1-
h_T(p) 
), $$ that is at most 
 $$ \ll\prod_{B<p\mid d}(1-
h_T(p) 
)^{-2}
\prod_{B<p<M}(1-
h_T(p) 
).$$
To  treat the main  term  
sum $ \sum_{m\mid \mathcal P} \lambda_m^-
h_T(m) 
$ we use~\cite[Equation~(6.40)]{MR2061214}, which is 
a more precise version of~\cite[Equation~(6.48)]{MR2061214} in the case of $\lambda_m^-$. Specifically,~\cite[Equation~(6.40)]{MR2061214}  states that  
$$\sum_{m\mid \mathcal P} \lambda_m^-
h_T(m) 
>(1-\mathrm e^{\beta-s}K^{10}) \prod_{p\leq z} (1-h_T(p) 
),$$ where $ \beta=1+9\kappa$ and $s=( \log y)/(\log z)$. 
In our case one has $s=   \epsilon \theta/v$ and a simple calculation shows that our definition of $v$ ensures that $1-\mathrm e^{\beta-s}K^{10} \geq 1/2$, thus,   $$ \sum_{m\mid \mathcal P} \lambda_m^-
h_T(m) 
\gg \prod_{p\leq z }(1-h_T(p) 
).$$
Injecting our estimates in~\eqref{eq:canipleasesithere?} gives
$$\sum_{a\in \mathcal A } \chi_T(a) f(c_a)\gg  M   \prod_{p\leq M^v  }(1-
h_T(p) 
) \sum_{  d \leq M^{v_0}  } f(d)
h_T(d) 
+o(M\mathcal T)+O(M^{1-\xi/2} ),$$ where $$\mathcal T=   
\prod_{\substack{ B<p \leq M^v  } } (1-
h_T(p) 
)\sum_{1\leq d \leq M^{v_0}  } f(d) 
h_T(d) 
\prod_{\substack{ B< p\mid d } } (1-
h_T(p) 
)^{-2} .$$ Letting  $c(p)=(1-
h_T(p) 
)^{-2} -1$ and applying 
Lemma~\ref{lem:damnnotfinal} we  obtain $$\sum_{1\leq d \leq M^{v_0}  } f(d) 
h_T(d) 
\prod_{\substack{ B< p\mid d } } (1-
h_T(p) 
)^{-2} \ll \sum_{1\leq d \leq M^{v_0}  } f(d) 
h_T(d) 
.$$  
This leads to 
$$\sum_{a\in \mathcal A} \chi_T(a) f(c_a)\gg  M   \prod_{p\leq M^v  }(1-
h_T(p) 
) \sum_{1\leq d \leq M^{v_0}  } f(d)
h_T(d) 
 +O(M^{1-\xi/2} ).
 $$ 
Since $h_T(p) \in [0,1)$ for $p>M^v$ and using that $v\leq 1$, 
  the product over $p\leq M^v$ is at least $   \prod_{p\leq M   }(1-
h_T(p) 
)$. It thus remains to prove 
$$ \sum_{1\leq d \leq M^{v_0}  } f(d)
h_T(d) 
\gg  \sum_{1\leq d \leq M  } f(d)
h_T(d) 
.$$ Using the fact that $f$ and $h_T$ are both multiplicative we can write 
$$ \sum_{1\leq d \leq M  } f(d)
h_T(d) 
= \sum_{\substack{ 1\leq b \leq M \\P^+(b)\leq M^{v_0}}  } f(b)
h_T(b) 
 \sum_{\substack{ 1\leq c \leq M/b  \\P^-(c)> M^{v_0}}  } f(c) h_T(c)  $$ and it suffices to prove that the sum 
over $c$ is bounded independently of $M$. 
We apply Lemma~\ref{lem:lord of all fevers and plague} to get the upper
bound
\[ \sum_{\substack{ 1\leq c \leq M/b  \\P^-(c)> M^{v_0}}  } f(c)
h_T(c) 
\ll \exp\left(\sum_{M^{v_0}<p\leq M}f(p)
h_T(p) 
\right).\]
Recall that $f(p)\leq A$ and $
h_T(p) 
\leq B/p$, so the sum over $p$ is 
\[\ll \sum_{M^{v_0}<p\leq M} \frac{1}{p}=O(1),\] thus, concluding the proof.

\end{document}